\documentclass[14 pt, a4paper]{article}
\usepackage{latexsym}
\usepackage{floatpag}
\usepackage{multirow}
\usepackage{graphicx}
\usepackage{amsmath}
\usepackage{amssymb}
\usepackage{tikz}
\usepackage{caption}
\usepackage{tkz-euclide}
\usetikzlibrary{positioning}
\usepackage{centernot}
\usepackage{amsfonts}
\usepackage{changepage}
\usepackage{enumerate}
\usepackage[a4paper,left=2cm,right=2cm,top=2cm,bottom=2cm]{geometry}
\usepackage{diagbox}
\usepackage{authblk}
\usepackage{tabu}
\usepackage{amsthm}
\usepackage{hyperref}
\usepackage{float}
\usetikzlibrary{calc}
\usepackage{lipsum}

\makeatletter
\def\th@plain{%
  \thm@notefont{}% same as heading font
  \itshape % body font
}
\def\th@definition{%
  \thm@notefont{}% same as heading font
  \normalfont % body font
}
\makeatother

\newtheorem{theorem}{Theorem}
\newtheorem{col}{Corollary}

\newtheorem{lemma}{Lemma}

\newtheorem*{tuzconj}{Tuza's Conjecture (1986)}
\newtheorem{exmp}{Example}

\theoremstyle{definition}

\newtheorem{rem}{Remark}
\theoremstyle{plain}

\usepackage{graphicx} % Required for inserting images

\newcommand{\qedclaim}{
\hfill\scalebox{0.6}{$\blacksquare$}}

\newcommand\blfootnote[1]{%
  \begin{NoHyper}
  \renewcommand\thefootnote{}\footnote{#1}%
  \addtocounter{footnote}{-1}%
  \end{NoHyper}
}

\title{The asymptotic behaviour of $sat(n,\mathcal{F})$}
\author{Asier Calbet and Andrea Freschi}
\date{}

\begin{document}

\maketitle

\begin{abstract}

    \noindent For a family $\mathcal{F}$ of graphs, $sat(n,\mathcal{F})$ is the minimum number of edges in a graph $G$ on $n$ vertices which does not contain any of the graphs in $\mathcal{F}$ but such that adding any new edge to $G$ creates a graph in $\mathcal{F}$. For singleton families $\mathcal{F}$, Tuza conjectured that  $sat(n,\mathcal{F})/n$ converges and Truszczynski and Tuza discovered that either $sat(n,\mathcal{F})= \left(1-1/r\right)n+o(n)$ for some integer $r \geq 1$ or $ sat(n,\mathcal{F}) \geq n+o(n) $. This is often cited in the literature as the main progress towards proving Tuza's Conjecture. Unfortunately, the proof is flawed. We give a correct proof, which requires a novel construction. Moreover, for finite families $\mathcal{F}$, we completely determine the possible asymptotic behaviours of $sat(n,\mathcal{F})$ in the sparse regime $sat(n,\mathcal{F}) \leq n+o(n)$. Finally, we essentially determine which sequences of integers are of the form $\left(sat(n,\mathcal{F})\right)_{n \geq 0}$ for some (possibly infinite) family $\mathcal{F}$. 
\end{abstract}

\blfootnote{AC: School of Mathematical Sciences, Queen Mary University of London, United Kingdom, \tt{a.calbetripodas@qmul.ac.uk}}
\blfootnote{AF: School of Mathematics, University of Birmingham, United Kingdom, \tt{axf079@bham.ac.uk}}

\section{Introduction}

Given a family $\mathcal{F}$ of graphs, we say that a graph $G$ is \emph{$\mathcal{F}$-saturated} if it is maximally $\mathcal{F}$-free, meaning $G$ does not contain a graph in~$\mathcal{F}$ but every graph $G+e$ obtained by adding a new edge $e$ to $G$ contains a graph in~$\mathcal{F}$. We then define $sat(n,\mathcal{F})$ (for integers $n \geq 0$ and non-empty families $\mathcal{F}$ with $e(F) \geq 1$ for all $F \in \mathcal{F}$) to be the minimum number of edges in an $\mathcal{F}$-saturated graph on $n$ vertices. If $\mathcal{F}=\{F\}$ is a singleton, we say that a graph is $F$-saturated rather than $\mathcal{F}$-saturated and write $sat(n,F)$ instead of $sat(n,\mathcal{F})$. \newline

The saturation problem is to determine, or at least estimate, $sat(n,\mathcal{F})$, where we usually think of $\mathcal{F}$ as fixed and $n$ as large. The saturation problem is dual to the Tur\'an forbidden subgraph problem of determining the extremal number $ex(n,\mathcal{F})$, defined to be the maximum number of edges in an $\mathcal{F}$-saturated graph on $n$ vertices. However, much less is known for the saturation problem than for the Tur\'an problem and the saturation number behaves more irregularly than the extremal number. See \cite{Survey} for a survey of the saturation problem. \newline

\noindent K\'aszonyi and Tuza showed that $sat(n,\mathcal{F})=O(n)$ for all families $\mathcal{F}$ (see Theorem 1 in~\cite{KT}). Tuza conjectured that if $\mathcal{F}$ is a singleton, $sat(n,\mathcal{F}) / n$ converges (see Conjecture 10 in \cite{tuzaconj}). 

\begin{tuzconj}\label{conj:asymp}
For every graph $F$, the limit $\lim\limits_{n\to\infty} sat(n,F)/n$ exists.
\end{tuzconj}

Despite being arguably the main open problem in the area of graph saturation, little progress has been made towards settling this conjecture. Pikhurko made some progress in the negative direction by showing there is an infinite family $\mathcal{F}$ of graphs such that $sat(n,\mathcal{F}) / n$ diverges (see Section 2.2 in \cite{Pik1}) and later improved this by showing there is a family $\mathcal{F}$ of four graphs such that $sat(n,\mathcal{F}) / n$ diverges (see Example 5 in \cite{Pik2}). This was then further improved by Chakraborti and Loh to a family $\mathcal{F}$ of three graphs (see Theorem 1.8 in \cite{counterfam}). \newline

This is in keeping with a recurring theme in  graph saturation, whereby the larger one allows $\mathcal{F}$ to be, the more irregularly $sat(n,\mathcal{F})$ can behave. With this in mind, in this paper we distinguish three cases: the case where $\mathcal{F}=\{F\}$ is a singleton, the case where $\mathcal{F}$ is finite and the case where $\mathcal{F}$ is possibly infinite, or in other words, where $\mathcal{F}$ is arbitrary. \newline

For arbitrary families $\mathcal{F}$, we obtain the following result, which essentially determines which sequences of integers are of the form $\left(sat(n,\mathcal{F})\right)_{n \geq 0}$ for some family $\mathcal{F}$.

\begin{theorem}\label{arbfam}
There exist sequences $(a_m)_{m \geq 0}$ and $(b_m)_{m \geq 0}$ of non-negative integers with $ \exp\left[m \log m + O(m)\right] \geq a_m \geq b_m \geq \exp\left[m \log m + O(m \log \log m)\right]$ such that the following hold.

    \begin{enumerate}    
        \item For every family $\mathcal{F}$ of graphs, the following hold.
        \begin{enumerate}
            \item $sat(n,\mathcal{F})=O(n)$.
            \item For every integer $m \geq 0$, the set $\{n \in \mathbb{Z}_{\geq 0} : sat(n,\mathcal{F})=m\}$ is the union of at most $a_m$ intervals.
            \item $sat(n,\mathcal{F})$ either tends to infinity or is eventually constant.
        \end{enumerate}
        \item Let $(s_n)_{n \geq 0}$ be a sequence of integers such that the following hold.
        \begin{enumerate}
            \item $s_n=O(n)$.
            \item For every large enough integer $m$, the set $\{n \in \mathbb{Z}_{\geq 0} : s_n=m\}$ is the union of at most $b_m$ intervals.
            \item  $s_n \to \infty$ as $n \to \infty$.
        \end{enumerate}
        Then there exists a family $\mathcal{F}$ of graphs such that $sat(n,\mathcal{F})=s_n$ for large enough $n$.
        \item For every integer $m \geq 0$, there exists a finite family $\mathcal{F}$ of graphs such that  $sat(n,\mathcal{F})=m$ for large enough~$n$.
    \end{enumerate}
\end{theorem}

\newtheorem*{repth1}{Theorem \ref{arbfam}}

\begin{rem}
    Recall that part 1 (a) was proved by K\'aszonyi and Tuza. We include it only for comparison with part 2 (a).
\end{rem}

\begin{rem}
    Part 1 (c) in fact follows from part 1 (b) and is only included for comparison with parts 2 (c) and~3.
\end{rem}

Condition 2 (b) is rather mild. Indeed, we obtain the following two easy corollaries of part 2.

\begin{col}\label{incrcor}
    Let $(s_n)_{n \geq 0}$ be a sequence of integers such that the following hold.
        \begin{enumerate}[(a)]
            \item $s_n=O(n)$.
            \item $s_{n+1} \geq s_n$ for large enough $n$.
            \item  $s_n \to \infty$ as $n \to \infty$.
        \end{enumerate}
        Then there exists a family $\mathcal{F}$ of graphs such that $sat(n,\mathcal{F})=s_n$ for large enough $n$.
\end{col}

\newtheorem*{repincrcor}{Corollary \ref{incrcor}}

It follows from Corollary \ref{incrcor} that $sat(n,\mathcal{F})$ can tend to infinity arbitrarily slowly.

\begin{col}\label{sizecor}

    Let $(s_n)_{n \geq 0}$ be a sequence of integers with

    $$ \frac{\log n}{\log \log n} + \omega\left(\frac{\log n \log \log \log n}{(\log \log n)^2}\right)\leq s_n = O(n). $$
    
    Then there exists a family $\mathcal{F}$ of graphs such that $sat(n,\mathcal{F})=s_n$ for large enough $n$.
    
\end{col}

\newtheorem*{repsizecor}{Corollary \ref{sizecor}}

It follows from Corollary \ref{sizecor} that $sat(n,\mathcal{F})$ can behave very irregularly. Indeed, we obtain the following easy corollary of Corollary \ref{sizecor}.

\begin{col}\label{AP}
    For every non-empty compact set $K$ of non-negative real numbers, there exists a family $\mathcal{F}$ of graphs such that the set of accumulation points of $\left(sat(n,\mathcal{F})/n\right)_{n \geq 1}$ is $K$. 
\end{col}

\newtheorem*{repAP}{Corollary \ref{AP}}

 Recall that an accumulation point of a sequence is the limit of a convergent subsequence. Since $sat(n,\mathcal{F})=O(n)$, the sequence $\left(sat(n,\mathcal{F})/n\right)_{n \geq 1}$ has at least one accumulation point for every family $\mathcal{F}$. Tuza's Conjecture is equivalent to the sequence $\left(sat(n,F)/n\right)_{n \geq 1}$ having only one accumulation point for every graph $F$. Hence, for families $\mathcal{F}$, one can view the set of accumulation points of $\left(sat(n,\mathcal{F})/n\right)_{n \geq 1}$ (and in particular its cardinality) as a measure of the extent to which the analogue of Tuza's Conjecture for $\mathcal{F}$ is false. \newline

Note that the conditions on $K$ in Corollary \ref{AP} are trivially necessary: $K$ must be non-empty since $\left(sat(n,\mathcal{F})/n\right)_{n \geq 1}$ has at least one accumulation point, $K$ must be bounded since $sat(n,\mathcal{F})=O(n)$, $K$ must be closed since every set of accumulation points is closed and $K$ must contain only non-negative real numbers since $sat(n,\mathcal{F}) \geq 0$ for all integers $n \geq 0$. \newline

Let us call a family $\mathcal{F}$ of graphs \emph{quasi-finite} if for every tree $T$ there is a constant $C_T$ such that every graph in $\mathcal{F}$ has at most $C_T$ connected components isomorphic to $T$. Note that every finite family is quasi-finite. Quasi-finite families behave much like finite families for our purposes, hence the name. For quasi-finite families~$\mathcal{F}$, we obtain the following result, which completely determines the possible asymptotic behaviours of $sat(n,\mathcal{F})$ in the sparse regime $sat(n,\mathcal{F}) \leq n+o(n)$. 

\begin{theorem}\label{quasfinfam}  \leavevmode
\begin{enumerate}
    \item Let $\mathcal{F}$ be a quasi-finite family of graphs. Then for each integer $r \geq 1$, there is a set $S_r$ of residue classes modulo $r$ and for each residue class $C \in S_r$, there is an integer $k_C$, such that the following hold.

        \begin{enumerate}
        
        \item For every integer $r \geq 1$ and residue class $C \in S_r$,
        
        $$ sat(n,\mathcal{F})= \frac{(r-1)n+k_C}{r} $$
        
        for large enough $n \in C \setminus \bigcup_{q<r, D \in S_q} D$.
        
        \item For $n \in \mathbb{Z}_{\geq 0} \setminus \bigcup_{r \geq 1, C \in S_r} C$,
        
        $$ sat(n,\mathcal{F}) \geq n+o(n) . $$
        
        \item For $r \in \{1,2,3,4,5,7\}$, $S_r$ contains either none or all of the residue classes modulo $r$.
        
        \end{enumerate}

    \item For each integer $r \geq 1$, let $S_r$ be a set of residue classes modulo $r$. Suppose that for $r \in \{1,2,3,4,5,7\}$, $S_r$ contains either none or all of the residue classes modulo $r$. Then there exists a quasi-finite family $\mathcal{F}$ of graphs such that the following hold. 

        \begin{enumerate}
     
        \item For every integer $r \geq 1$, 
        
        $$sat(n,\mathcal{F})= \left(1-\frac{1}{r}\right)n+O(1)$$
        
        for $n \in \left(\bigcup_{C \in S_r} C \right) \setminus \bigcup_{q<r, C \in S_q} C$.
        
        \item For $n \in \mathbb{Z}_{\geq 0} \setminus \bigcup_{r \geq 1, C \in S_r} C$,
        
        $$ sat(n,\mathcal{F}) \geq n+o(n)  . $$ 
        
        \item If there are only finitely many integers $r \geq 1$ with $\left(\bigcup_{C \in S_r} C \right) \setminus \bigcup_{q<r, C \in S_q} C \neq \emptyset$, $\mathcal{F}$ is in fact finite.
        
        \end{enumerate}

\end{enumerate}
\end{theorem}

\newtheorem*{repth2}{Theorem \ref{quasfinfam}}

\begin{rem}
    The converse of part 2 (c) is also true - in part 1, if $\mathcal{F}$ is in fact finite, then there are only finitely many integers $r \geq 1$ with $\left(\bigcup_{C \in S_r} C \right) \setminus \bigcup_{q<r, C \in S_q} C \neq \emptyset$ - but this is a deeper result whose proof we postpone to a future paper. 
\end{rem}

For singleton families $\mathcal{F}=\{F\}$, Truszczynski and Tuza discovered that either $sat(n,F)= \left(1-1/r\right)n+o(n)$ for some integer $r \geq 1$ or $ sat(n,F) \geq n+o(n) $, thus verifying Tuza's Conjecture in the sparse regime $sat(n,F) \leq n+o(n)$ (see Theorem 2 in \cite{TrusTuza}). They also characterised the graphs $F$ with $sat(n,F)= \left(1-1/r\right)n+o(n)$ for each integer $r \geq 1$ (see Theorem 3 in \cite{TrusTuza}). This is often cited in the literature as the main progress towards settling Tuza's Conjecture in the positive direction. Unfortunately, the proof in \cite{TrusTuza} is flawed (see section \ref{sin} for more details). We give a correct proof using a novel construction, thus confirming both the result and the characterisation.

\begin{theorem}\label{sinfam}

For every graph $F$, either there is an integer $r \geq 1$ and an integer $k_C$ for each residue class $C$ modulo $r$ such that 

$$ sat(n,F)= \frac{(r-1)n+k_C}{r} $$
        
for every residue class $C$ modulo $r$ and large enough $n \in C$ or

$$ sat(n,F) \geq n+o(n) . $$ 

The former occurs if and only if there is a tree $T$ such that every graph $T+e$ obtained by adding a new edge $e$ to $T$ and every graph $2T+e$ obtained from the disjoint union $2T$ of two copies of $T$ by adding an edge $e$ between them contains a connected component of $F$ that is not contained in $T$ and $r$ is the minimum number of vertices in such a tree.

\end{theorem}

\newtheorem*{repth3}{Theorem \ref{sinfam}}

\begin{rem}
    It easily follows from the characterisation that for every integer $r \geq 1$ there exists a graph $F$ with $sat(n,F)=\left(1-1/r\right)n+O(1)$. Indeed, let $T$ be any tree with $r$ vertices and $F$ be the disjoint union of all graphs $T+e$ obtained by adding a new edge $e$ to $T$ and all graphs $2T+e$ obtained from $2T$ by adding an edge $e$ between the two copies of $T$.
\end{rem}

\begin{rem}
    It follows from the characterisation that for every integer $r \geq 1$ there is a finite collection of graphs such that for every  graph $F$, whether or not $sat(n,F)=\left(1-1/r\right)n+O(1)$ depends only on which graphs in the collection are connected components of $F$. For every integer $r \geq 1$ one can compute this finite collection and dependence and hence give an explicit description of the graphs $F$ with $sat(n,F)=\left(1-1/r\right)n+O(1)$, but the complexity of this description increases with $r$ (see Corollary 4 in \cite{TrusTuza}).
\end{rem}

The rest of the paper is organised as follows. We first prove Theorem \ref{arbfam} and deduce Corollaries 1 through 3 in section \ref{arb}. We then prove Theorem \ref{quasfinfam} in section \ref{fin}. Finally, we discuss the error in \cite{TrusTuza} and prove Theorem \ref{sinfam} in section \ref{sin}.

\section{Arbitrary families}\label{arb}

In this section we prove Theorem \ref{arbfam} and deduce Corollaries 1 through 3. To prove Theorem \ref{arbfam}, we will need two lemmas. The first is a fact about set differences of unions of intervals that we will need in the proof of part 1 (b) of Theorem \ref{arbfam}. Before stating and proving the lemma, we define intervals and introduce some notation. \newline

For our purposes, an \emph{interval} is a set of integers $I$ with the property that if $a \leq b \leq c$ are integers with $a,c \in I$ then $b \in I$. Equivalently, an interval is a set of integers that is either empty or of the form $[a,b]$ for some $a \in \{-\infty\} \cup \mathbb{Z}$ and $b \in \mathbb{Z} \cup \{\infty\}$ with $a \leq b$, where $[a,b]=\{c \in \mathbb{Z} : a \leq c \leq b\}$ and $-\infty \leq c \leq \infty$ for all $c \in \mathbb{Z}$.

\begin{lemma}\label{intervalsetdiff}
    Let $A$ be the union of $a$ intervals and $B$ be the union of $b$ intervals. Then $A \setminus B$ is the union of at most $a+b$ intervals.
\end{lemma}

\begin{proof}

We show that if $C$ is the union of at most $c$ intervals and $D$ is an interval then $C \setminus D$ is the union of at most $c+1$ intervals, from which the lemma follows by induction on $b$. We may assume $D$ is non-empty. Write $C=\bigcup_{i=1}^d C_i$, where the $C_i$ are intervals and $d \leq c$ is minimal. Then, since the union of two intersecting intervals is itself an interval, the $C_i$ are disjoint, as otherwise $C$ would be the union of fewer than $d$ intervals. \newline

It is easy to see that, for every integer $1 \leq i \leq d$, if $D\not\subseteq C_i$ then $C_i \setminus D$ is an interval. Furthermore, if $D\subseteq C_i$ then $C_i \setminus D$ is the union of at most two intervals. Since the $C_i$ are disjoint and $D$ is non-empty, there is at most one integer $1 \leq i \leq d$ such that $D\subseteq C_i$. It follows that $C \setminus D=\bigcup_{i=1}^d \left( C_i\setminus D \right)$ is the union of at most $d+1\le c+1$ intervals, as required.

\end{proof}

We say a family of graphs is an \emph{antichain} if no graph in the family contains another. The second lemma asserts the existence of an antichain containing many graphs with $m$ edges and large minimum degree for every large enough integer $m$. We will need such an antichain to prove part 2 of Theorem \ref{arbfam}.

\begin{lemma}\label{antichain}

    There exists an antichain $\mathcal{A}$ of graphs (up to isomorphism) such that the following hold, where $\mathcal{A}_m$ is the family of graphs in $\mathcal{A}$ with $m$ edges for each integer $m \geq 0$.

\begin{enumerate}
    \item $|\mathcal{A}_m| = \exp\left[m \log m + O(m \log \log m)\right]. $
    \item $|G| = m / \log m + O(1)$ for every $G \in \mathcal{A}_m $.
    \item $0 < \delta(G) = \Omega(\log m) $ for every $G \in \mathcal{A}_m$.
\end{enumerate}
    
\end{lemma}

 The proof is a non-constructive counting argument. In the proof (and at one point in the proof of Theorem \ref{arbfam}) we will need to estimate certain binomial coefficients. All estimates can be obtained by writing $\binom{a}{b}=a!/b!(a-b)!$ and then applying Stirling's formula $N!=\exp\left[N \log N - N + \log N / 2 + O(1) \right]$. The calculations involved are lengthy but straightforward, so we omit them.

\begin{proof}
    Let $C$ be a constant, large enough for the proof to go through. Crucially, all implied constants in the proof are independent of $C$. We recursively obtain families $\mathcal{A}_m$ of graphs up to isomorphism for each integer $m \geq C$ such that the following hold.

    \begin{enumerate}[(a)]
        \item $|\mathcal{A}_m| = \left\lfloor \exp\left(m \log m - C m \log \log m\right) \right\rfloor $.
        \item  Every graph $G \in \mathcal{A}_m $ has the following properties.
           \begin{enumerate}[(i)]
                \item $|G| = \lfloor m / \log m  \rfloor $.
                \item $e(G)=m $.
                \item $\delta(G) \geq \log m / C$.
                \item $G \not \supseteq H$ for every integer $C \leq \ell < m$ and $H \in \mathcal{A}_\ell$. 
        \end{enumerate}
    \end{enumerate}
    
    Then $\mathcal{A}=\bigcup_{m \geq C} \ \mathcal{A}_m$ is the desired antichain. Suppose $m \geq C$ is an integer and we have obtained $\mathcal{A}_\ell$ for every integer $C \leq \ell < m$. We then obtain $\mathcal{A}_m$ as follows. Let $n=\lfloor m / \log m  \rfloor$ and $\mathcal{B}_m$ be the family of all graphs on $n$ labeled vertices with $m$ edges. \newline
    
    \emph{Claim 1:} The proportion of graphs in $\mathcal{B}_m$ with a vertex of degree less than $\log m / C$ is $o(1)$. \newline

    \emph{Proof:} For every vertex $v$ and integer $0 \leq d = o(\log m)$, the proportion of graphs in $\mathcal{B}_m$ with $d(v)=d$ is 

    $$\frac{\displaystyle \binom{n-1}{d}\binom{\binom{n-1}{2}}{m-d}}{\displaystyle \binom{\binom{n}{2}}{m}} = \frac{1}{ m^{ 2+o(1)}} \ . $$

    Hence, for every vertex $v$ and integer $0 \leq d < \log m / C$, the proportion of graphs in $\mathcal{B}_m$ with $d(v)=d$ is at most $1 / m^{3/2}$, provided $C$ is large enough. There are $n=O(m/ \log m)$ vertices $v$ and $O(\log m)$ integers $0 \leq d < \log m / C$, so by a union bound, the proportion of graphs in $\mathcal{B}_m$ with a vertex of degree less than $\log m / C$ is $O(1 / m^{1/2})$, which is $o(1)$.\qedclaim\newline

    \emph{Claim 2:} The proportion of graphs in $\mathcal{B}_m$ containing a graph in $\mathcal{A}_\ell$ for some integer $C \leq \ell < m$ is $o(1)$. \newline

    \emph{Proof:} Since there are at most $m$ integers $C \leq \ell < m$, by a union bound, it is sufficient to show that for every integer $C \leq \ell < m$, the proportion of graphs in $\mathcal{B}_m$ containing a graph in $\mathcal{A}_\ell$ is $o(1/m)$. Let $C \leq \ell < m$ be an integer. For every graph $H \in \mathcal{A}_\ell$, since $|H|=\lfloor \ell / \log \ell \rfloor \leq \ell / \log \ell$, the number of copies of $H$ on $n$ labeled vertices is at most $n^{|H|} \leq \exp\left(\ell \log m / \log \ell\right)$ and, since $e(H)=\ell$, the proportion of graphs in $\mathcal{B}_m$ containing a given copy of $H$ is

    $$\frac{\displaystyle \binom{\binom{n}{2}-\ell}{m-\ell}}{\displaystyle \binom{\binom{n}{2}}{m}} = \exp\left[-\ell \log m + O(\ell \log \log m)\right] . $$

    Since $|\mathcal{A}_\ell| \leq \exp\left(\ell \log \ell - C \ell \log \log \ell \right)$, by a union bound, the proportion of graphs in $\mathcal{B}_m$ containing a graph in $\mathcal{A}_l$ is at most

    $$ \exp\left[\ell \log \ell - \ell \log m - C \ell \log \log \ell + O(\ell \log \log m) + \ell \log m / \log \ell  \right]  .$$
    
    Suppose first that $ \ell \leq m^{1/4}$. Then $ \ell \log \ell \leq \ell \log m / 4$, $- C \ell \log \log \ell \leq 0$, $ O( \ell \log \log m) \leq \ell \log m / 4 $ and $ \ell \log m / \log \ell \leq \ell \log m / 4$, provided $C$ is large enough, since $m > \ell \geq C$. Hence the proportion of graphs in $\mathcal{B}_m$ containing a graph in $\mathcal{A}_{ \ell }$ is at most $\exp(- \ell \log m /4) \leq 1/m^2$, which is $o(1/m)$, provided $C$ is large enough, since $ \ell \geq C$. \newline
    
    Now suppose $\ell>m^{1/4}$. Then $ \ell \log \ell - \ell \log m \leq 0$, $- C \ell \log \log \ell \leq -2 C \ell \log \log m /3 $ and $O( \ell \log \log m) + \ell \log m / \log \ell = O( \ell \log \log m) \leq C \ell \log \log m /3$, provided $C$ is large enough, since $m \geq C$. Hence the proportion of graphs in $\mathcal{B}_m$ containing a graph in $\mathcal{A}_\ell$ is at most $\exp(-C \ell \log \log m/3) \leq \exp(-m^{1/4} \log \log m)$, which is $o(1/m)$, provided $C$ is large enough. \qedclaim\newline

    Every graph in $\mathcal{B}_m$ has properties (i) and (ii) and by Claims 1 and 2, $\left(1+o(1)\right)|\mathcal{B}_m|$ of the graphs in $\mathcal{B}_m$ also have properties (iii) and (iv). Since every graph in $\mathcal{B}_m$ is isomorphic to at most $n!$ graphs in $\mathcal{B}_m$, it follows that the number of graphs up to isomorphism with properties (i) through (iv) is at least

    $$\frac{1+o(1)}{n!} \ |\mathcal{B}_m|=\frac{1+o(1)}{n!} \ \binom{\binom{n}{2}}{m}=\exp\left[m \log m + O(m \log \log m) \right] , $$

    which is at least $\left\lfloor \exp\left(m \log m - C m \log \log m\right) \right\rfloor $, provided $C$ is large enough. Hence, we obtain a family $\mathcal{A}_m$ of graphs up to isomorphism satisfying (a) and (b).
    
\end{proof}

We are now ready to prove Theorem \ref{arbfam}.

\begin{repth1}
There exist sequences $(a_m)_{m \geq 0}$ and $(b_m)_{m \geq 0}$ of non-negative integers with $ \exp\left[m \log m + O(m)\right] \geq a_m \geq b_m \geq \exp\left[m \log m + O(m \log \log m)\right]$ such that the following hold.

    \begin{enumerate}    
        \item For every family $\mathcal{F}$ of graphs, the following hold.
        \begin{enumerate}
            \item $sat(n,\mathcal{F})=O(n)$.
            \item For every integer $m \geq 0$, the set $\{n \in \mathbb{Z}_{\geq 0} : sat(n,\mathcal{F})=m\}$ is the union of at most $a_m$ intervals.
            \item $sat(n,\mathcal{F})$ either tends to infinity or is eventually constant.
        \end{enumerate}
        \item Let $(s_n)_{n \geq 0}$ be a sequence of integers such that the following hold.
        \begin{enumerate}
            \item $s_n=O(n)$.
            \item For every large enough integer $m$, the set $\{n \in \mathbb{Z}_{\geq 0} : s_n=m\}$ is the union of at most $b_m$ intervals.
            \item  $s_n \to \infty$ as $n \to \infty$.
        \end{enumerate}
        Then there exists a family $\mathcal{F}$ of graphs such that $sat(n,\mathcal{F})=s_n$ for large enough $n$.
        \item For every integer $m \geq 0$, there exists a finite family $\mathcal{F}$ of graphs such that  $sat(n,\mathcal{F})=m$ for large enough~$n$.
    \end{enumerate}
\end{repth1}

\begin{proof}

    We first define the sequences $(a_m)_{m \geq 0}$ and $(b_m)_{m \geq 0}$ and prove the inequalities $ \exp\left[m \log m + O(m)\right] \geq a_m \geq b_m \geq \exp\left[m \log m + O(m \log \log m)\right]$. For each integer $m \geq 0$, let $\mathcal{G}_m$ be the family of graphs (up to isomorphism) with $m$ edges and no isolated vertices and let $a_m = 2 \sum_{\ell=0}^m |\mathcal{G}_\ell| $ (note that the families $\mathcal{G}_m$ are finite, since every graph in $\mathcal{G}_m$ has at most $2m$ vertices). By Lemma \ref{antichain}, there is an antichain $\mathcal{A}$ of graphs (up to isomorphism) with the properties stated in the lemma. For each integer $m \geq 0$, let $b_m=|\mathcal{A}_m|$, where $\mathcal{A}_m$ is the family of graphs in $\mathcal{A}$ with $m$ edges, as before. Then $b_m \geq \exp\left[m \log m + O(m \log \log m)\right]$ by part 1 of Lemma \ref{antichain} and by part 3 of Lemma \ref{antichain}, every graph in $\mathcal{A}_m$ has no isolated vertices, so $a_m \geq |\mathcal{G}_m| \geq b_m$. \newline

    Since every graph in $\mathcal{G}_m$ has at most $2m$ vertices, we can associate to it the graph on $2m$ vertices obtained by adding isolated vertices. Hence $|\mathcal{G}_m|$ is the number of graphs (up to isomorphism) with $2m$ vertices and $m$ edges, which is at most the number of graphs on $2m$ labeled vertices with $m$ edges, which is 

    $$ \binom{\binom{2m}{2}}{m} = \exp\left[m \log m + O(m)\right] . $$

    So

    $$a_m = 2 \sum_{l=0}^m |\mathcal{G}_l| \leq 2 \sum_{\ell=0}^m \exp\left[\ell \log \ell + O(\ell)\right]  = \exp\left[m \log m + O(m)\right] .  $$

    We now prove part 1. Let $\mathcal{F}$ be a family of graphs. Part (a) was proved by K\'aszonyi and Tuza (see Theorem 1 in \cite{KT}). We now prove part (b). For each graph $H$ with no isolated vertices, let $H_k$ be the graph obtained by adding $k$ isolated vertices to $H$, for each integer $k \geq 0$, and let $I_H=\{k \in \mathbb{Z}_{\geq 0} : \text{$H_k$ is $\mathcal{F}$-saturated} \}$. \newline

\begin{figure}[H]
\centering
\begin{tikzpicture}

    \draw (3,-4) ellipse (0.75 and 0.75);
    \node[anchor=center] at (3,-4) {$H$};
    \filldraw[black] (5,-4) circle (0.75pt);
    \filldraw[black] (4.875+0.5,0.28125-4) circle (0.75pt); 
    \filldraw[black] (4.125+0.5,0.28125-4) circle (0.75pt);
    \filldraw[black] (4.125+0.5,-4.28125) circle (0.75pt);
    \filldraw[black] (4.875+0.5,-4.28125) circle (0.75pt);
    \draw [decorate,decoration={brace,amplitude=5pt,raise=5pt}] (5+0.5,-4.28125) -- (5-0.5,-4.28125) node[midway,below=15pt] {$k$};

\end{tikzpicture}
\caption{The graph $H_k$.}\label{fig:H_k}
\end{figure}
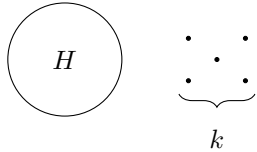

    \emph{Claim 1:} For every graph $H$ with no isolated vertices, $I_H$ is the union of two intervals. \newline
    
    \emph{Proof:} We in fact show that $I_H \cap \mathbb{Z}_{\geq 2}$ is an interval, from which the claim follows, since $I_H \cap \{0,1\}$ is an interval (since every subset of $\{0,1\}$ is an interval). Let $2 \leq k \leq \ell \leq m$ be integers such that $H_k$ and $H_m$ are $\mathcal{F}$-saturated. We show that $H_\ell$ is $\mathcal{F}$-saturated and hence that $I_H \cap \mathbb{Z}_{\geq 2}$ is an interval. Since $H_\ell \subseteq H_m$ and $H_m$ is $\mathcal{F}$-free, so is $H_\ell$. It remains to show that every graph $H_\ell+vw$ obtained by adding a new edge $vw$ to $H_\ell$ contains a graph in $\mathcal{F}$. \newline
    
    Suppose first that $v,w \in H$. Then, since $H_k + vw \subseteq H_\ell + vw$ and $H_k + vw$ contains a graph in $\mathcal{F}$, so does $H_\ell + vw$. Now suppose $v \in H$ and $w \not \in H$. Let $w' \in H_k$ be one of the isolated vertices (note $k \geq 1$). Then, since $H_k + vw' \subseteq H_l + vw$ and $H_k + vw'$ contains a graph in $\mathcal{F}$, so does $H_l + vw$. Finally, suppose $v,w \not \in H$. Let $v',w' \in H_k$ be distinct isolated vertices (note $k \geq 2$). Then, since $H_k + v'w' \subseteq H_\ell + vw$ and $H_k + v'w'$ contains a graph in $\mathcal{F}$, so does $H_\ell + vw$.\qedclaim\newline

    For each integer $m \geq 0$, let $S_m$ be the set of integers $n \geq 0$ for which there exists an 
    $\mathcal{F}$-saturated graph $G$ with $n$ vertices and $m$ edges. Note that every graph $G$ is of the form $H_k$ for some graph $H$ with no isolated vertices and integer $k \geq 0$. Hence, $S_m = \bigcup_{H \in \mathcal{G}_m} I_H + |H|$, which is the union of $2|\mathcal{G}_m|$ intervals by Claim~1 (where $S+n=\{s+n : s \in S\}$ for every set $S$ of integers and integer $n$). So, for every integer $m \geq 0$, $\{n \in \mathbb{Z}_{\geq 0} : sat(n,\mathcal{F})=m\} = S_m \setminus \bigcup_{\ell<m} S_\ell$ is the union of at most $a_m$ intervals  by Lemma \ref{intervalsetdiff}, proving part (b). \newline 

    Part (c) in fact follows from part (b). Indeed, by part (b), for every integer $m \geq 0$, the set $\{n \in \mathbb{Z}_{\geq 0} : sat(n,\mathcal{F})=m\}$ is a finite union of intervals. If for all $m$ all of these intervals are finite, the sets $\{n \in \mathbb{Z}_{\geq 0} : sat(n,\mathcal{F})=m\}$ are all finite, which is equivalent to $sat(n,\mathcal{F})$ tending to infinity. On the other hand, if for some $m$ one of these intervals is infinite, $sat(n,\mathcal{F})$ is eventually constant. \newline

    Next, we prove part 2. Let $(s_n)_{n \geq 0}$ be a sequence of integers satisfying the conditions in part 2. By condition~(a), there is a constant $C_1>0$ such that $s_n \leq C_1(n+1)$ for every integer $n \geq 0$. By condition~(b) and parts 2 and 3 of Lemma \ref{antichain}, there is a constant $C_2$ such that for every integer $m \geq C_2$, the set $\{n \in \mathbb{Z}_{\geq 0} : s_n=m\}$ is the union of at most $|\mathcal{A}_m|$ intervals and for every graph $G \in  \mathcal{A}_m$, $m/C_1 - 1 \geq |G|+2 $ and $\delta(G) \geq 2C_1+3$. \newline
    
    For each integer $m \geq C_2$, let $\mathcal{I}_m$ be a collection of intervals with $\{n \in \mathbb{Z}_{\geq 0} : s_n=m\} = \bigcup_{I \in \mathcal{I}_m} I$ and  $|\mathcal{I}_m| \leq |\mathcal{A}_m|$. We may assume that every $I \in \mathcal{I}_m$ is non-empty and it follows from condition (c) that every $I \in \mathcal{I}_m$ is finite, so every interval $I \in \mathcal{I}_m$ is of the form $[a_I,b_I]$ for some integers $0 \leq a_I \leq b_I$. Since $|\mathcal{I}_m| \leq |\mathcal{A}_m|$, we can pick a different graph $H_I \in \mathcal{A}_m$ for each $I \in \mathcal{I}_m$. For each integer $m \geq C_2$ and interval $I \in \mathcal{I}_m$, let $\mathcal{F}_I$ be the family containing the following graphs. 

    \begin{enumerate}
        \item Every graph on $a_I$ vertices obtained from $H_I$ by first adding a new edge and then isolated vertices.
        \item Every graph on $a_I$ vertices obtained from $H_I$ by first adding a new vertex, adjacent to a unique vertex in $H_I$, and then isolated vertices.
        \item The graph on $a_I$ vertices obtained from $H_I$ by adding an isolated edge and isolated vertices.
        \item The graph on $b_I+1$ vertices obtained from $H_I$ by adding isolated vertices.
    \end{enumerate}

\begin{figure}[H]
\centering
\begin{tikzpicture}

    \draw (-2.5,0) ellipse (0.75 and 0.75);
    \node[anchor=center] at (-2.5,0) {$H_I$};
    \filldraw[black] (-0.375-2.5,0.4871) circle (0.75pt);
    \filldraw[black] (0.375-2.5,0.1624) circle (0.75pt);
    \draw  (-0.375-2.5,0.4871) -- (0.375-2.5,0.1624);
    \filldraw[black] (1.5-2,0) circle (0.75pt);
    \filldraw[black] (1.875-2,0.28125) circle (0.75pt); 
    \filldraw[black] (1.125-2,0.28125) circle (0.75pt);
    \filldraw[black] (1.125-2,-0.28125) circle (0.75pt);
    \filldraw[black] (1.875-2,-0.28125) circle (0.75pt);
    \draw [decorate,decoration={brace,amplitude=5pt,raise=5pt}] (1.5-2+0.5,-0.28125) -- (1.5-2-0.5,-0.28125) node[midway,below=10pt] {$a_I-|H_I|$};
    \node[anchor=north] at (-1.5,-1.5) {Type 1};

    \draw (3,0) ellipse (0.75 and 0.75);
    \node[anchor=center] at (3,0) {$H_I$};
    \filldraw[black] (3+0.28125,-0.365325) circle (0.75pt);
    \filldraw[black] (3+1.125,-0.1624) circle (0.75pt);
    \draw (3+0.28125,-0.365325) -- (3+1.125,-0.1624);
    \filldraw[black] (5.25,0) circle (0.75pt);
    \filldraw[black] (3.75+1.875,0.28125) circle (0.75pt); 
    \filldraw[black] (3.75+1.125,0.28125) circle (0.75pt);
    \filldraw[black] (3.75+1.125,-0.28125) circle (0.75pt);
    \filldraw[black] (3.75+1.875,-0.28125) circle (0.75pt);
    \draw [decorate,decoration={brace,amplitude=5pt,raise=5pt}] (5.75,-0.28125) -- (4.75,-0.28125) node[midway,below=10pt] {$a_I-|H_I|-1$};
    \node[anchor=north] at (4,-1.5) {Type 2};

    \draw (-3,-4) ellipse (0.75 and 0.75);
    \node[anchor=center] at (-3,-4) {$H_I$};
    \filldraw[black] (-1.5,0.5625-4) circle (0.75pt);
    \filldraw[black] (-1.5,-0.5625-4) circle (0.75pt);
    \draw (-1.5,0.5625-4) -- (-1.5,-0.5625-4);
    \filldraw[black] (0,-4) circle (0.75pt);
    \filldraw[black] (1.875-1.5,0.28125-4) circle (0.75pt); 
    \filldraw[black] (1.125-1.5,0.28125-4) circle (0.75pt);
    \filldraw[black] (1.125-1.5,-0.28125-4) circle (0.75pt);
    \filldraw[black] (1.875-1.5,-0.28125-4) circle (0.75pt);
    \draw [decorate,decoration={brace,amplitude=5pt,raise=5pt}] (0.5,-0.28125-4) -- (-0.5,-0.28125-4) node[midway,below=10pt] {$a_I-|H_I|-2$};
    \node[anchor=north] at (-1.5,-5.5) {Type 3};

     \draw (3,-4) ellipse (0.75 and 0.75);
    \node[anchor=center] at (3,-4) {$H_I$};
    \filldraw[black] (5,-4) circle (0.75pt);
    \filldraw[black] (4.875+0.5,0.28125-4) circle (0.75pt); 
    \filldraw[black] (4.125+0.5,0.28125-4) circle (0.75pt);
    \filldraw[black] (4.125+0.5,-4.28125) circle (0.75pt);
    \filldraw[black] (4.875+0.5,-4.28125) circle (0.75pt);
    \draw [decorate,decoration={brace,amplitude=5pt,raise=5pt}] (5.5,-4.28125) -- (4.5,-4.28125)node[midway,below=10pt] {$b_I-|H_I|+1$};
    \node[anchor=north] at (4,-5.5) {Type 4};

\end{tikzpicture}
\caption{The four types of graphs in $\mathcal{F}_I$.}\label{fig:4types}
\end{figure}
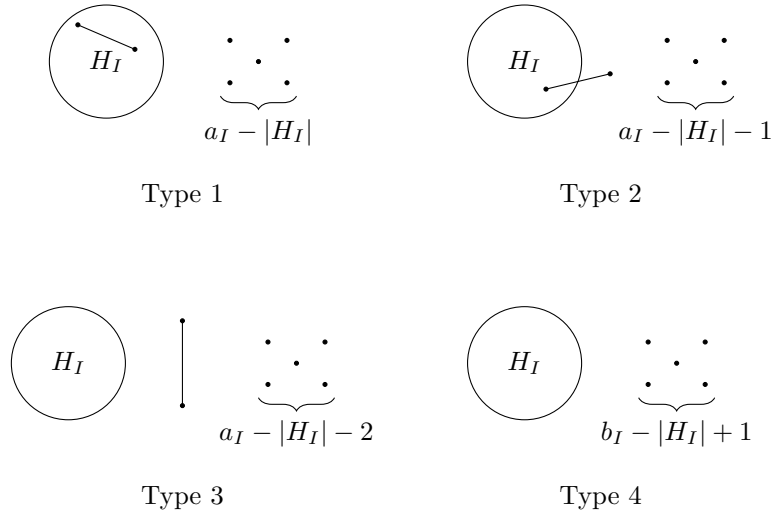

Note that $C_1 (a_I +1) \geq s_{a_I} = m $, so $a_I \geq m/C_1 -1 \geq |H_I| + 2$. Let $\mathcal{F}=\bigcup_{\substack{m \geq C_2,  I \in \mathcal{I}_m}} \mathcal{F}_I$.\newline 
    
    \emph{Claim 2:} For every integer $m \geq C_2$ and interval $I \in \mathcal{I}_m$, every graph $G$ with $|G| \in I$ obtained by adding isolated vertices to $H_I$ is $\mathcal{F}$-saturated. \newline

    \emph{Proof:} We first show that $G$ is $\mathcal{F}$-free. Note that for every integer $m' \geq C_2$ and interval $I' \in \mathcal{I}_{m'}$, every graph in $\mathcal{F}_{I'}$ contains $H_{I'}$. Hence, if $G$ contains a graph in $\mathcal{F}_{I'}$, it contains $H_{I'}$. But by part 3 of Lemma \ref{antichain}, $H_{I'}$ has no isolated vertices, so this can only happen if $H_I$ contains $H_{I'}$. Since $\mathcal{A}$ is an antichain, $\mathcal{A}_{m'}$ is disjoint from $\mathcal{A}_m$ if $m' \neq m$ and $H_{I'} \neq H_I$ if $I' \in \mathcal{I}_m \setminus \{I\}$, this can only happen if $m'=m$ and $I'=I$. \newline
    
    But every graph in $\mathcal{F}_I$ of the first, second or third type has more edges than $G$ and the graph in $\mathcal{F}_I$ of the fourth type has more vertices than $G$, so $G$ is $\mathcal{F}$-free. It remains to show that every graph $G+vw$ obtained by adding a new edge $vw$ to $G$ contains a graph in $\mathcal{F}$. If $v,w \in H_I$, $G+vw$ contains a graph in $\mathcal{F}_I$ of the first type, if $v \in H_I$ and $w \not \in H_I$, $G+vw$ contains a graph in $\mathcal{F}_I$ of the second type, and if $v,w \not \in H_I$, $G+vw$ contains the graph in $\mathcal{F}_I$ of the third type. \qedclaim\newline

    Since  $a_I \geq |H_I|$ for every integer $m \geq C_2$ and interval $I \in \mathcal{I}_m$, by Claim 2 and condition (c), there is an $\mathcal{F}$-saturated graph with $n$ vertices and $s_n$ edges for every large enough integer $n$. The following claim completes the proof of part 2. \newline

    \emph{Claim 3:} For large enough $n$, every $\mathcal{F}$-saturated graph $G$ on $n$ vertices not of the form in Claim 2 has more than $s_n$ edges. \newline

    \emph{Proof:} We show that $\delta(G) \geq 2C_1+2$ and hence that $e(G) \geq (C_1+1)n > C_1(n+1) \geq s_n$ if $n$ is large enough. Let $v \in G$. If $v$ is adjacent to every other vertex, $d(v) = n-1 \geq 2C_1+2$ if $n$ is large enough, so suppose there is a vertex $w \neq v$ not adjacent to $v$. Then $G+vw$ contains a graph in $\mathcal{F}_I$ for some integer $m \geq C_2$ and interval $I \in \mathcal{I}_m$. Since every graph in $\mathcal{F}_I$ contains $H_I$ and has at least $a_I$ vertices, it follows that $G+vw \supseteq H_I$ and $n \geq a_I$. \newline
    
    Now suppose for the sake of contradiction that $G$ contains a copy of $H_I$. Then, since $n \geq a_I$, the copy of $H_I$ must be induced, for otherwise $G$ would contain a graph in $\mathcal{F}_I$ of the first type, no vertex in the copy can be adjacent to a vertex outside the copy, for otherwise $G$ would contain a graph in $\mathcal{F}_I$ of the second type, there can be no edge outside the copy, for otherwise $G$ would contain the graph in $\mathcal{F}_I$ of the third type, and $n \leq b_I$, for otherwise $G$ would contain the graph in $\mathcal{F}_I$ of the fourth type. But then $G$ is of the form in Claim 2, a contradiction. \newline

    Hence $G \not \supseteq H_I$ but $G+vw \supseteq H_I$, so the edge $vw$, and hence the vertex $v$, must belong to a copy of $H_I$ in $G+vw$. Since $\delta(H_I) \geq 2C_1+3$, it follows that $d(v) \geq 2C_1+2$. \qedclaim\newline

    Finally, we prove part 3. Let $m \geq 0$ be an integer. Let $H$ be any graph with $m$ edges and $\mathcal{F}$ be the family containing the following graphs.

    \begin{enumerate}
        \item Every graph $H+e$ obtained by adding a new edge $e$ to $H$. 
        \item Every graph obtained by adding a new vertex to $H$ and joining it to a vertex in $H$.
        \item The graph obtained by adding an isolated edge to $H$.
    \end{enumerate}

\begin{figure}[H]
\centering
\begin{tikzpicture}

    \begin{scope}[shift={(0,-0.25)}]
    \draw (-2.5,0) ellipse (0.75 and 0.75);
    \node[anchor=center] at (-2.5,0) {$H$};
    \filldraw[black] (-0.375-2.5,0.4871) circle (0.75pt);
    \filldraw[black] (0.375-2.5,0.1624) circle (0.75pt);
    \draw  (-0.375-2.5,0.4871) -- (0.375-2.5,0.1624);
    \node[anchor=north] at (-2.5,-1) {Type 1};
    \end{scope}
    
    \begin{scope}[shift={(0.9,-0.25)}]
    \draw (0,0) ellipse (0.75 and 0.75);
    \node[anchor=center] at (0,0) {$H$};
    \filldraw[black] (0.28125,-0.365325) circle (0.75pt);
    \filldraw[black] (1.125,-0.1624) circle (0.75pt);
    \draw (0.28125,-0.365325) -- (1.125,-0.1624);
    \node[anchor=north] at (0.25,-1) {Type 2};
    \end{scope}
    
    \begin{scope}[shift={(2,-0.25)}]
    \draw (2.5,0) ellipse (0.75 and 0.75);
    \node[anchor=center] at (2.5,0) {$H$};
    \filldraw[black] (4,0.4) circle (0.75pt);
    \filldraw[black] (4,-0.4) circle (0.75pt);
    \draw (4,0.4) -- (4,-0.4);
    \node[anchor=north] at (3,-1) {Type 3};
    \end{scope}

\end{tikzpicture}
\caption{The three types of graphs in $\mathcal{F}$.}\label{fig:3types}
\end{figure}
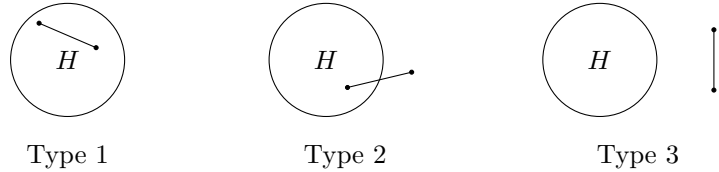

    Note that $\mathcal{F}$ is finite. We first show that every graph $G$ obtained by adding isolated vertices to $H$ is $\mathcal{F}$-saturated. Indeed, $G$ is $\mathcal{F}$-free, since every graph in $\mathcal{F}$ has more edges than $G$, and every graph $G+vw$ obtained by adding a new edge $vw$ to $G$ contains a graph in $\mathcal{F}$ - of the first type if $v,w \in H$, of the second type if $v \in H$ and $w \not \in H$ and of the third type if $v,w \not \in H$. Hence $sat(n,\mathcal{F}) \leq m$ for large enough $n$. \newline
    
    Now suppose $G$ is an $\mathcal{F}$-saturated graph on $n$ vertices. If $n$ is large enough, $G$ cannot be complete, for otherwise it would contain the graph in $\mathcal{F}$ of the third type. Hence we can add a new edge to $G$ to obtain a graph containing a graph in $\mathcal{F}$. Since every graph in $\mathcal{F}$ has $m$+1 edges, it follows that $G$ has at least $m$ edges. So $sat(n,\mathcal{F}) = m$ for large enough $n$.
    
\end{proof}

We now deduce Corollaries 1 through 3.

\begin{repincrcor}

        Let $(s_n)_{n \geq 0}$ be a sequence of integers such that the following hold.
        \begin{enumerate}[(a)]
            \item $s_n=O(n)$.
            \item $s_{n+1} \geq s_n$ for large enough $n$.
            \item  $s_n \to \infty$ as $n \to \infty$.
        \end{enumerate}
        Then there exists a family $\mathcal{F}$ of graphs such that $sat(n,\mathcal{F})=s_n$ for large enough $n$.

\end{repincrcor}

\begin{proof}

By part 2 of Theorem \ref{arbfam}, we just need to check that $(s_n)_{n \geq 0}$ satisfies condition (b). By condition (b) in Corollary \ref{incrcor}, there is a constant $C \geq 0$ such that $s_{n+1} \geq s_n$ for every integer $n > C$. We show that the set $\{n \in \mathbb{Z}_{\geq 0} : s_n=m\}$ is an interval for every integer $m>\max_{0 \leq n \leq C } \ s_n$, which implies that $(s_n)_{n \geq 0}$ satisfies condition (b), since $b_m \geq 1$ for large enough $m$. Let $m>\max_{0 \leq n \leq C } \ s_n$ and $0 \leq a \leq b \leq c$ be integers with $s_a=m=s_c$. Then $a>C$, so $m=s_a \leq s_{a+1} \leq s_{a+2} \leq \cdots \leq s_b \leq \cdots \leq s_c=m$, so $s_b=m$.

\end{proof}

\begin{repsizecor}

    Let $(s_n)_{n \geq 0}$ be a sequence of integers with

    $$ \frac{\log n}{\log \log n} + w\left(\frac{\log n \log \log \log n}{(\log \log n)^2}\right)\leq s_n = O(n). $$
    
    Then there exists a family $\mathcal{F}$ of graphs such that $sat(n,\mathcal{F})=s_n$ for large enough $n$.
    
\end{repsizecor}

\begin{proof}
    
By part 2 of Theorem \ref{arbfam}, we just need to check that $(s_n)_{n \geq 0}$ satisfies condition (b). We show that for every large enough integer $m$, $s_n>m$ for every integer $n \geq b_m$. Hence every element of the set $\{n \in \mathbb{Z}_{\geq 0} : s_n=m\}$ is smaller than $b_m$, so the set has size at most $b_m$ and hence is the union of at most $b_m$ (singleton) intervals. Let $m \geq 0$ and $n \geq b_m$ be integers. Then $n \geq b_m \geq \exp\left[m \log m + O(m \log \log m)\right]$, so

$$ s_n \geq \frac{\log n}{\log \log n} + w\left(\frac{\log n \log \log \log n}{\left(\log \log n\right)^2}\right) $$

$$ \geq  \frac{m \log m + O(m \log \log m)}{\log \left[m \log m + O(m \log \log m) \right]} + w\left(\frac{\left[m \log m + O(m \log \log m) \right] \log \log \left[m \log m + O(m \log \log m) \right]}{(\log \left[m \log m + O(m \log \log m) \right])^2}\right) $$

$$ = m + w\left(\frac{m \log \log m}{\log m} \right) , $$

since $\log x/ \log \log x$ and $\log x \log \log \log x / (\log \log x)^2$ are eventually increasing functions of $x$. Hence $s_n>m$ if $m$ is large enough.

\end{proof}

\begin{repAP}
    For every non-empty compact set $K$ of non-negative real numbers, there exists a family $\mathcal{F}$ of graphs such that the set of accumulation points of $\left(sat(n,\mathcal{F})/n\right)_{n \geq 1}$ is $K$.
\end{repAP}

\begin{proof}

We first construct a countable set $S$ of real numbers with $\overline{S}=K$ (where $\overline{S}$ is the closure of $S$). For all integers $m \geq 1$ and $n$, let $S_{m,n}=\emptyset$ if $K \cap [n/m,(n+1)/m] = \emptyset$ and let $S_{m,n}=\{x\}$ for some $x \in K \cap [n/m,(n+1)/m]$ otherwise (where $[n/m,(n+1)/m]=\{x \in \mathbb{R} : n/m \leq x \leq (n+1)/m \}$). Let $S=\bigcup_{m \geq 1, n} S_{m,n}$. Then $S$ is countable and $S \subseteq K$, so $\overline{S} \subseteq K$, since $K$ is closed. It remains to show that $K \subseteq \overline{S}$. Let $k \in K$. Then for every integer $m \geq 1$, there is an integer $n$ such that $k \in [n/m,(n+1)/m]$. But then $K \cap [n/m,(n+1)/m] \neq \emptyset$, so there exists $x \in S \cap [n/m,(n+1)/m]$. Then $|x-k| \leq 1/m$. Since $m$ is arbitrary, it follows that $k \in \overline{S}$. \newline

Since $K$ is non-empty, bounded and contains only non-negative real numbers, the same is true for $S$. Since $S$ is non-empty and countable, there is a partition $\mathbb{Z}_{\geq 0}=\bigcup_{s \in S} P_s$ where every part $P_s$ is infinite. For each integer $ n \geq 0$, let $s_n=\left\lceil sn+\sqrt{n} \right\rceil$, where $s$ is the unique element of $S$ with $n \in P_s$. Then $(s_n)_{n \geq 0}$ is a sequence of integers with
$\sqrt{n} \leq s_n = O(n)$, since $S$ is bounded and contains only non-negative real numbers, so by Corollary~\ref{sizecor}, there is a family $\mathcal{F}$ of graphs such that $sat(n,\mathcal{F})=s_n$ for large enough $n$. \newline

Let $A$ be the set of accumulation points of $\left(sat(n,\mathcal{F})/n\right)_{n \geq 1}$. We show that $A=K$. For every $s \in S$, $sat(n,\mathcal{F})/n = s + O\left(1/\sqrt{n}\right)$ for $n \in P_s$. Since $P_s$ is infinite, it follows that $s \in A$. Hence $S \subseteq A$, so $K = \overline{S} \subseteq A$, since $A$ is closed. It remains to show that $A \subseteq K$. Let $a \in A$. Then there are arbitrarily large integers $n$ with $sat(n,\mathcal{F})/n$ arbitrarily close to $a$. Since for every $n$, $sat(n,\mathcal{F})/n = s + O(1/\sqrt{n})$ for some $s \in S$, it follows that there are $s \in S$ arbitrarily close to $a$, or in other words that $a \in \overline{S}=K$.

\end{proof}

\section{Quasi-finite families}\label{fin}

In this section we prove Theorem \ref{quasfinfam}. We first introduce some notation. Given graphs $G$ and $H$, we write $G \cup H$ for the disjoint union of $G$ and $H$. Given a graph $G$ and an integer $k \geq 0$, we write $kG$ for the disjoint union of $k$ copies of $G$. Given a tree $T$, a subtree $T' \subseteq T$, vertices $v \in T'$ and $w \in T$ and an integer $k \geq 0$, let $\left(T' \cup kT \right) + v\mathbf{w}$ be the tree obtained from $T' \cup kT$ by joining $v$ to the $k$ copies of $w$. 

\begin{figure}[H]
\centering
\scalebox{1}{
\begin{tikzpicture}

    \draw (0,0) ellipse (0.75 and 0.75);
    \node[anchor=center] at (0,0) {$T'$};
    \filldraw[black] (0,-0.375) circle (0.75pt);
    \node[anchor=west] at (0,-0.34) {$v$};

    \begin{scope}[shift={(-2.25/2,0)}]
    \draw (4.5,-3) ellipse (0.75 and 0.75);
    \node[anchor=center] at (4.5,-3) {$T$};
    \filldraw[black] (4.5,-2.625) circle (0.75pt);
    \node[anchor=west] at (4.5,-2.625) {$w$};
    \end{scope}
    
    \draw  ($ (-2.25,-2.625)+(2.25/2,0) $) -- (0,-0.375);
    \draw  ($ (-4.5,-2.625)+(2.25/2,0) $) -- (0,-0.375);
    \draw  (0,-0.375) -- (4.5-2.25/2,-2.625);
    
    \begin{scope}[shift={(2.25/2,0)}]
    \draw (-2.25,-3) ellipse (0.75 and 0.75);
    \node[anchor=center] at (-2.25,-3) {$T$};
    \filldraw[black] (-2.25,-2.625) circle (0.75pt);
    \node[anchor=east] at (-2.25,-2.625) {$w$};

    \draw (-4.5,-3) ellipse (0.75 and 0.75);
    \node[anchor=center] at (-4.5,-3) {$T$};
    \filldraw[black] (-4.5,-2.625) circle (0.75pt);
    \node[anchor=east] at (-4.5,-2.625) {$w$};
    \end{scope}

    \node[anchor=center] at (2.25/2,-3) {$\cdots$};
    
    \draw [decorate,decoration={brace,amplitude=10pt,raise=5pt}] (2.25*1.5+0.75,-3.75) -- (-2.25*1.5-0.75,-3.75) node[midway,below=20pt] {$k$};

\end{tikzpicture}}
\caption{The tree $\left(T' \cup kT\right)+v\mathbf{w}$.}\label{fig:vectortree}
\end{figure}
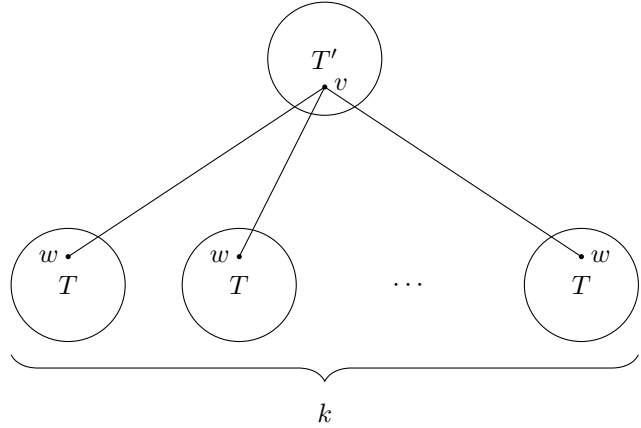

We will need the following lemma in the proof of part 1 (c) of Theorem \ref{quasfinfam}.

\begin{lemma}\label{expl}
For every tree $T$ with $|T| \in \{1,2,3,4,5,7\}$ there is a subtree $T' \subseteq T$ such that $|T'|$ and $|T|$ are coprime and for all vertices $v \in T'$ and $w \in T$, $\left(T' \cup kT \right) + v\mathbf{w}$ contains a copy of $T$ that contains $v$ for some integer $k \geq 0$.
\end{lemma}

\begin{proof}

We first list all trees $T$ with $|T| \in \{1,2,3,4,5,7\}$. \newline

\begin{figure}[H]
\centering
\scalebox{1.5}{
\begin{tikzpicture}

    \filldraw[black] (-3,3) circle (0.75pt);
    \draw  (-3,3) ellipse (0.075 and 0.075);
    \node[anchor=north] at (-3,2.75) {$T_{1,1}$};

    \draw  (-2,3) -- (-1.5,3);
    \filldraw[black] (-2,3) circle (0.75pt);
    \draw  (-2,3) ellipse (0.075 and 0.075);
    \filldraw[black] (-1.5,3) circle (0.75pt);
    \node[anchor=north] at (-1.75,2.75) {$T_{2,1}$};

    \draw  (-0.5,3) -- (0,3);
    \draw  (0,3) -- (0.5,3);
    \filldraw[black] (-0.5,3) circle (0.75pt);
    \filldraw[black] (0,3) circle (0.75pt);
    \draw  (0,3) ellipse (0.075 and 0.075);
    \filldraw[black] (0.5,3) circle (0.75pt);
    \node[anchor=north] at (0,2.75) {$T_{3,1}$};

    \draw (1.933,3) -- (1.5,2.75);
    \draw (1.933,3) -- (1.933,3.5);
    \draw (1.933,3) -- (2.366,2.75);
    \node[anchor=north] at (1.933,2.5) {$T_{4,1}$};
    \filldraw[black] (1.5,2.75) circle (0.75pt);
    \filldraw[black] (1.933,3) circle (0.75pt);
    \draw (1.933,3) ellipse (0.075 and 0.075);
    \filldraw[black] (1.933,3.5) circle (0.75pt);
    \filldraw[black] (2.366,2.75) circle (0.75pt);

    \draw  (3.366+0.5,3) -- (3.366,3);
    \draw  (4.366,3) -- (3.366+0.5,3);
    \draw  (4.366,3) -- (4.366+0.5,3);
    \filldraw[black] (3.366,3) circle (0.75pt);
    \filldraw[black] (3.866,3) circle (0.75pt);
    \draw  (3.866,3) ellipse (0.075 and 0.075);
    \filldraw[black] (4.366,3) circle (0.75pt);
    \filldraw[black] (4.366+0.5,3) circle (0.75pt);
    \node[anchor=north] at (4.116,2.75) {$T_{4,2}$};

    \draw  (-3,1) -- (-3,1.5);
    \draw  (-3,1) -- (-3,0.5);
    \draw  (-3,1) -- (-2.5,1);
    \draw  (-3,1) -- (-3.5,1);
    \node[anchor=center] at (-3,0) {$T_{5,1}$};
    \filldraw[black] (-3,1) circle (0.75pt);
    \draw  (-3,1) ellipse (0.075 and 0.075);
    \filldraw[black] (-3,1.5) circle (0.75pt);
    \filldraw[black] (-3,0.5) circle (0.75pt);
    \filldraw[black] (-2.5,1) circle (0.75pt);
    \filldraw[black] (-3.5,1) circle (0.75pt);

    \draw  (-1.5,1.3535) -- (-1.1464,1);
    \draw  (-1.5,1-0.3535) -- (-1.1464,1);
    \draw  (-1.1464,1) -- (-0.6464,1);
    \draw  (-0.6464,1) -- (-0.1464,1);
    \node[anchor=center] at (-0.8232,0) {$T_{5,2}$};
    \filldraw[black] (-1.5,1.3535) circle (0.75pt);
    \filldraw[black] (-1.5,1-0.3535) circle (0.75pt);
    \filldraw[black] (-1.1464,1) circle (0.75pt);
    \draw  (-1.1464,1) ellipse (0.075 and 0.075);
    \filldraw[black] (-0.6464,1) circle (0.75pt);
    \filldraw[black] (-0.1464,1) circle (0.75pt);

    \draw  (0.8536,1) -- (1.3536,1);
    \draw  (1.3536,1) -- (1.8536,1);
    \draw  (1.8536,1) -- (2.3536,1);
    \draw  (2.3536,1) -- (2.8536,1);
    \node[anchor=north] at (1.8536,0.75) {$T_{5,3}$};
    \filldraw[black] (0.8536,1) circle (0.75pt);
    \filldraw[black] (1.3536,1) circle (0.75pt);
    \filldraw[black] (1.8536,1) circle (0.75pt);
    \draw  (1.8536,1) ellipse (0.075 and 0.075);
    \filldraw[black] (2.3536,1) circle (0.75pt);
    \filldraw[black] (2.8536,1) circle (0.75pt);

    \draw  (3.8536,1.25) -- (4.2866,1);
    \draw  (3.8536,0.75) -- (4.2866,1);
    \draw  (4.2866,1.5) -- (4.2866,1);
    \draw  (4.2866,0.5) -- (4.2866,1);
    \draw  (4.7196,1.25) -- (4.2866,1);
    \draw  (4.7196,0.75) -- (4.2866,1);
    \node[anchor=center] at (4.2866,0) {$T_{7,1}$};
    \filldraw[black] (3.8536,1.25) circle (0.75pt);
    \filldraw[black] (3.8536,0.75) circle (0.75pt);
    \filldraw[black] (4.2866,1.5) circle (0.75pt);
    \filldraw[black] (4.2866,1) circle (0.75pt);
    \draw (4.2866,1) ellipse (0.075 and 0.075);
    \filldraw[black] (4.2866,0.5) circle (0.75pt);
    \filldraw[black] (4.7196,1.25) circle (0.75pt);
    \filldraw[black] (4.7196,0.75) circle (0.75pt);

    \filldraw[black] (-2.5,-1.5) circle (0.75pt);
    \filldraw[black] (-2,-1.5) circle (0.75pt);
    \draw (-3,-1.5) -- (-3.4045,0.2934-1.5);
    \draw (-3,-1.5) -- (-3.4045,-0.2934-1-0.5);
    \draw (-3,-1.5) -- (-2.8455,-1.4755-0.5);
    \draw (-3,-1.5) -- (-2.8455,-0.5245-0.5);
    \draw  (-3,-1.5) -- (-2.5,-1.5);
    \draw  (-2.5,-1.5) -- (-2,-1.5);
    \node[anchor=center] at (-2.7023,-2.5) {$T_{7,2}$};
    \filldraw[black] (-3,-1.5) circle (0.75pt);
    \draw (-3,-1.5) ellipse (0.075 and 0.075);
    \filldraw[black] (-3.4045,0.2934-1.5) circle (0.75pt);
    \filldraw[black] (-3.4045,-0.2934-1.5) circle (0.75pt);
    \filldraw[black] (-2.8455,-0.5245-0.5) circle (0.75pt);
    \filldraw[black] (-2.8455,-1.4755-0.5) circle (0.75pt);

    \draw (1.933-2.5,3-4.5-0.25) -- (1.5-2.5,2.75-4.5-0.25);
    \draw (1.933-2.5,3-4.5-0.25) -- (1.933-2.5,3.5-4.5-0.25);
    \draw (1.933-2.5,3-4.5-0.25) -- (2.366-2.5,2.75-4.5-0.25);
    \draw (1.933-2.5,-0.5-0.25) -- (1.933-2.5,-1-0.25);
    \draw (1.067-2.5,2.5-4.5-0.25) -- (1.5-2.5,2.75-4.5-0.25);
    \draw (2.799-2.5,2.5-4.5-0.25) -- (2.366-2.5,2.75-4.5-0.25);
    \node[anchor=center] at (1.933-2.5,-2.5) {$T_{7,3}$};
    \filldraw[black] (1.5-2.5,2.75-4.5-0.25) circle (0.75pt);
    \filldraw[black] (1.933-2.5,3-4.5-0.25) circle (0.75pt);
    \draw (1.933-2.5,3-4.5-0.25) ellipse (0.075 and 0.075);
    \filldraw[black] (1.933-2.5,3.5-4.5-0.25) circle (0.75pt);
    \filldraw[black] (1.933-2.5,-0.5-0.25) circle (0.75pt);
    \filldraw[black] (2.366-2.5,2.75-4.5-0.25) circle (0.75pt);
    \filldraw[black] (1.067-2.5,2.5-4.5-0.25) circle (0.75pt);
    \filldraw[black] (2.799-2.5,2.5-4.5-0.25) circle (0.75pt);

    \draw  (2.007,-1.5) -- (1.6535,-1.8535);
    \draw  (2.007,-1.5) -- (2.3605,-1.8535);
    \draw (2.3605,-1.8535) -- (2.714,-2.207);
    \draw (1.6535,-1.8535) -- (1.3,-2.207);
    \draw (2.007,-1.5) -- (1.6535,-1.1465);
    \draw (2.007,-1.5) -- (2.3605,-1.1465);
    \node[anchor=center] at (2.007,-2.5) {$T_{7,4}$};
    \filldraw[black] (1.3,-2.207) circle (0.75pt);
    \filldraw[black] (1.6535,-1.8535) circle (0.75pt);
    \filldraw[black] (2.007,-1.5) circle (0.75pt);
    \draw (2.007,-1.5) ellipse (0.075 and 0.075);
    \filldraw[black] (2.714,-2.207) circle (0.75pt);
    \filldraw[black] (2.3605,-1.8535) circle (0.75pt);
    \filldraw[black] (1.6535,-1.1465) circle (0.75pt);
    \filldraw[black] (2.3605,-1.1465) circle (0.75pt); 

    \draw  (3.8605,-1.5) -- (3.3605,-1.5);
    \draw  (3.8605,-1.5) -- (4.3605,-1.5);
    \draw (3.8605,-1.5) -- (3.8605,-2);
    \draw (3.8605,-1.5) -- (3.8605,-1);
    \draw (4.3605,-1.5) -- (4.8605,-1.5);
    \draw (4.8605,-1.5) -- (5.3605,-1.5);
    \node[anchor=center] at (4.3605,-2.5) {$T_{7,5}$};
    \draw  (3.8605,-1.5) ellipse (0.075 and 0.075);
    \filldraw[black] (3.3605,-1.5) circle (0.75pt);
    \filldraw[black] (3.8605,-1.5) circle (0.75pt);
    \filldraw[black] (4.3605,-1.5) circle (0.75pt);
    \filldraw[black] (4.8605,-1.5) circle (0.75pt);
    \filldraw[black] (5.3605,-1.5) circle (0.75pt);
    \filldraw[black] (3.8605,-2) circle (0.75pt);
    \filldraw[black] (3.8605,-1) circle (0.75pt);

    \draw (1.933-2.5-2-0.433,3-4.5-2-0.25) -- (1.5-2.5-2-0.433,2.75-4.5-2-0.25);
    \draw (1.933-2.5-2-0.433,3-4.5-2-0.25) -- (1.933-2.5-2-0.433,3.5-4.5-2-0.25);
    \draw (1.933-2.5-2-0.433,3-4.5-2-0.25) -- (2.366-2.5-2-0.433,2.75-4.5-2-0.25);
    \draw (1.067-2.5-2-0.433,2.5-4.5-2-0.25) -- (1.5-2.5-2-0.433,2.75-4.5-2-0.25);
    \draw (2.799-2.5-2-0.433,2.5-4.5-2-0.25) -- (2.366-2.5-2-0.433,2.75-4.5-2-0.25);
    \draw (2.799-2.5-2-0.433,2.5-4.5-2-0.25) -- (3.232-2.5-2-0.433,2.25-4.5-2-0.25);
    \node[anchor=center] at (1.933-2.5-2-0.433,2.5-4.5-2-0.5) {$T_{7,6}$};
    \filldraw[black] (1.5-2.5-2-0.433,2.75-4.5-2-0.25) circle (0.75pt);
    \filldraw[black] (1.933-2.5-2-0.433,3-4.5-2-0.25) circle (0.75pt);
    \draw  (1.933-2.5-2-0.433,3-4.5-2-0.25) ellipse (0.075 and 0.075);
    \filldraw[black] (1.933-2.5-2-0.433,3.5-4.5-2-0.25) circle (0.75pt);
    \filldraw[black] (2.366-2.5-2-0.433,2.75-4.5-2-0.25) circle (0.75pt);
    \filldraw[black] (1.067-2.5-2-0.433,2.5-4.5-2-0.25) circle (0.75pt);
    \filldraw[black] (2.799-2.5-2-0.433,2.5-4.5-2-0.25) circle (0.75pt);
    \filldraw[black] (3.232-2.5-2-0.433,2.25-4.5-2-0.25) circle (0.75pt);

    \draw  (-1.5+1.232-0.433,1.3535-4.5-0.25) -- (-1.1464+1.232-0.433,1-4.5-0.25);
    \draw  (-1.5+1.232-0.433,1-0.3535-4.5-0.25) -- (-1.1464+1.232-0.433,1-4.5-0.25);
    \draw  (-1.1464+1.232-0.433,1-4.5-0.25) -- (-0.6464+1.232-0.433,1-4.5-0.25);
    \draw  (-0.6464+1.232-0.433,1-4.5-0.25) -- (-0.1464+1.232-0.433,1-4.5-0.25);
    \draw  (-0.1464+1.232-0.433,1-4.5-0.25) -- (-0.1464+1.232+0.5-0.433,1-4.5-0.25);
    \draw  (-0.1464+1.232+0.5-0.433,1-4.5-0.25) -- (-0.1464+1.232+1-0.433,1-4.5-0.25);
    \node[anchor=center] at (-0.1464+1.232+1-0.433-1.125,2.5-4.5-2-0.5) {$T_{7,7}$};
    \draw  (-1.1464+1.232-0.433,1-4.5-0.25) ellipse (0.075 and 0.075);
    \filldraw[black] (-1.5+1.232-0.433,1.3535-4.5-0.25) circle (0.75pt);
    \filldraw[black] (-1.5+1.232-0.433,1-0.3535-4.5-0.25) circle (0.75pt);
    \filldraw[black] (-1.1464+1.232-0.433,1-4.5-0.25) circle (0.75pt);
    \filldraw[black] (-0.6464+1.232-0.433,1-4.5-0.25) circle (0.75pt);
    \filldraw[black] (-0.1464+1.232-0.433,1-4.50-0.25) circle (0.75pt);
    \filldraw[black] (-0.1464+1.232+0.5-0.433,1-4.5-0.25) circle (0.75pt);
    \filldraw[black] (-0.1464+1.232+1-0.433,1-4.5-0.25) circle (0.75pt);

    \draw (-0.1464+1.232+2-0.433,1-4.5-0.25) -- (-0.1464+1.732+2-0.433,1-4.5-0.25);
    \draw (-0.1464+1.232+2.5-0.433,1-4.5-0.25) -- (-0.1464+1.732+2.5-0.433,1-4.5-0.25);
    \draw (-0.1464+1.232+3-0.433,1-4.5-0.25) -- (-0.1464+1.732+3-0.433,1-4.5-0.25);
    \draw (-0.1464+1.232+3.5-0.433,1-4.5-0.25) -- (-0.1464+1.732+3.5-0.433,1-4.5-0.25);
    \draw (-0.1464+1.232+4-0.433,1-4.5-0.25) -- (-0.1464+1.732+4-0.433,1-4.5-0.25);
    \draw (-0.1464+1.232+4.5-0.433,1-4.5-0.25) -- (-0.1464+1.732+4.5-0.433,1-4.5-0.25);
    \node[anchor=center] at (-0.1464+2.732+2-0.433,2.5-4.5-2-0.5) {$T_{7,8}$};
    \filldraw[black] (-0.1464+1.232+2-0.433,1-4.5-0.25) circle (0.75pt);
    \filldraw[black] (-0.1464+1.732+2-0.433,1-4.5-0.25) circle (0.75pt);
    \filldraw[black] (-0.1464+2.232+2-0.433,1-4.5-0.25) circle (0.75pt);
    \filldraw[black] (-0.1464+2.732+2-0.433,1-4.5-0.25) circle (0.75pt);
    \draw (-0.1464+2.732+2-0.433,1-4.5-0.25) ellipse (0.075 and 0.075);
    \filldraw[black] (-0.1464+3.232+2-0.433,1-4.5-0.25) circle (0.75pt);
    \filldraw[black] (-0.1464+3.732+2-0.433,1-4.5-0.25) circle (0.75pt);
    \filldraw[black] (-0.1464+4.232+2-0.433,1-4.5-0.25) circle (0.75pt);

    \end{tikzpicture}}
    \end{figure}

    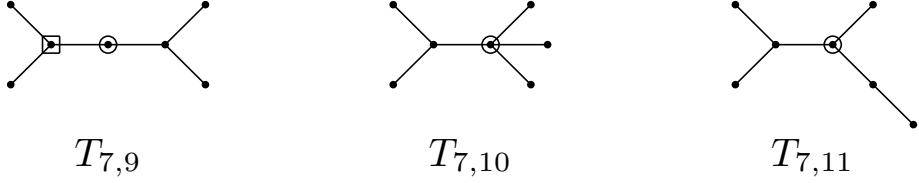
\begin{figure}[H]
    \centering
    \scalebox{1.5}{
    \begin{tikzpicture}

    \begin{scope}[shift={(0.8536,-0.25)}]

    \draw  (-1.5-1.8536-0.5,1.3535-6.5) -- (-1.1464-1.8536-0.5,1-6.5);
    \draw  (-1.5-1.8536-0.5,1-0.3535-6.5) -- (-1.1464-1.8536-0.5,1-6.5);
    \draw  (-1.1464-1.8536-0.5,1-6.5) -- (-0.6464-1.8536-0.5,1-6.5);
    \draw  (-0.6464-1.8536-0.5,1-6.5) -- (-0.1464-1.8536-0.5,1-6.5);
    \draw  (-0.1464-1.8536-0.5,1-6.5) -- (-0.1464-1.8536-0.5+0.3535,1-6.5+0.3535);
    \draw  (-0.1464-1.8536-0.5,1-6.5) -- (-0.1464-1.8536-0.5+0.3535,1-6.5-0.3535);
    \node[anchor=center] at (-0.6464-1.8536-0.5,-6.5) {$T_{7,9}$};
    \filldraw[black] (-1.5-1.8536-0.5,1.3535-6.5) circle (0.75pt);
    \filldraw[black] (-1.5-1.8536-0.5,1-0.3535-6.5) circle (0.75pt);
    \filldraw[black] (-1.1464-1.8536-0.5,1-6.5) circle (0.75pt);
    \begin{scope}[shift={(-1.1464-1.8536-0.5,1-6.5)}]
        \draw (-0.075,0.075) -- (0.075,0.075);
        \draw (0.075,0.075) -- (0.075,-0.075);
        \draw (0.075,-0.075) -- (-0.075,-0.075);
        \draw (-0.075,-0.075) -- (-0.075,0.075);
    \end{scope}
    \filldraw[black] (-0.6464-1.8536-0.5,1-6.5) circle (0.75pt);
    \draw (-0.6464-1.8536-0.5,1-6.5) ellipse (0.075 and 0.075);
    \filldraw[black] (-0.1464-1.8536-0.5,1-6.5) circle (0.75pt);
    \filldraw[black] (-0.1464-1.8536-0.5+0.3535,1-6.5+0.3535) circle (0.75pt);
    \filldraw[black] (-0.1464-1.8536-0.5+0.3535,1-6.5-0.3535) circle (0.75pt);

     \end{scope}

    \begin{scope}[shift={(1.5,-0.25)}]

    \draw  (-1.5+0.3535,1.3535-6.5) -- (-1.1464+0.3535,1-6.5);
    \draw  (-1.5+0.3535,1-0.3535-6.5) -- (-1.1464+0.3535,1-6.5);
    \draw  (-1.1464+0.3535,1-6.5) -- (-0.6464+0.3535,1-6.5);
    \draw  (-0.6464+0.3535,1-6.5) -- (-0.1464+0.3535,1-6.5);
    \draw   (-0.6464+0.3535,1-6.5) --  (-0.6464+0.3535+0.3535,1-6.5+0.3535);
    \draw   (-0.6464+0.3535,1-6.5) -- (-0.6464+0.3535+0.3535,1-6.5-0.3535);
    \node[anchor=center] at (-0.8232+0.3535,-6.5) {$T_{7,10}$};
    \filldraw[black] (-1.5+0.3535,1.3535-6.5) circle (0.75pt);
    \filldraw[black] (-1.5+0.3535,1-0.3535-6.5) circle (0.75pt);
    \filldraw[black] (-1.1464+0.3535,1-6.5) circle (0.75pt);
    \filldraw[black] (-0.6464+0.3535,1-6.5) circle (0.75pt);
    \draw (-0.6464+0.3535,1-6.5) ellipse (0.075 and 0.075);
    \filldraw[black] (-0.1464+0.3535,1-6.5) circle (0.75pt);
    \filldraw[black] (-0.6464+0.3535+0.3535,1-6.5+0.3535) circle (0.75pt);
    \filldraw[black] (-0.6464+0.3535+0.3535,1-6.5-0.3535) circle (0.75pt);

    \end{scope}

     \begin{scope}[shift={(2.5,-0.25)}]

    \draw  (-1.5+0.3535+2,1.3535-6.5) -- (-1.1464+0.3535+2,1-6.5);
    \draw  (-1.5+0.3535+2,1-0.3535-6.5) -- (-1.1464+0.3535+2,1-6.5);
    \draw  (-1.1464+0.3535+2,1-6.5) -- (-0.6464+0.3535+2,1-6.5);
    \draw   (-0.6464+0.3535+2,1-6.5) --  (-0.6464+0.3535+0.3535+2,1-6.5+0.3535);
    \draw   (-0.6464+0.3535+2,1-6.5) -- (-0.6464+0.3535+0.3535+2,1-6.5-0.3535);
    \draw   (-0.6464+0.3535+0.3535+2,1-6.5-0.3535) -- (-0.6464+0.3535+0.3535+0.3535+2,1-6.5-0.3535-0.3535);
    \node[anchor=center] at (-0.8232+0.3535+2,-6.5) {$T_{7,11}$};
    \filldraw[black] (-1.5+0.3535+2,1.3535-6.5) circle (0.75pt);
    \filldraw[black] (-1.5+0.3535+2,1-0.3535-6.5) circle (0.75pt);
    \filldraw[black] (-1.1464+0.3535+2,1-6.5) circle (0.75pt);
    \filldraw[black] (-0.6464+0.3535+2,1-6.5) circle (0.75pt);
    \draw (-0.6464+0.3535+2,1-6.5) ellipse (0.075 and 0.075);
    \filldraw[black] (-0.6464+0.3535+0.3535+2,1-6.5+0.3535) circle (0.75pt);
    \filldraw[black] (-0.6464+0.3535+0.3535+2,1-6.5-0.3535) circle (0.75pt);
    \filldraw[black] (-0.6464+0.3535+0.3535+0.3535+2,1-6.5-0.3535-0.3535) circle (0.75pt);

  \end{scope}
   
\end{tikzpicture}}
\caption{All trees $T$ with $|T| \in \{1,2,3,4,5,7\}$.}\label{fig:expl}
\end{figure}

For $T \not \in \{T_{7,10},T_{7,11}\}$ let $T'=T_{1,1}$. Then for all vertices $w \in T$, $\left(T' \cup kT \right) + v\mathbf{w}$ contains a copy of $T$ in which the copy of the circled vertex is $v$, where $v$ is the unique vertex of $T_{1,1}$ and $k$ is the degree of the circled vertex, unless $T=T_{7,9}$ and $w$ is a leaf, in which case $\left(T' \cup 3T \right) + v\mathbf{w}$ contains a copy of $T$ in which the copy of the boxed vertex is $v$. \newline

Let $T_6$ be the tree obtained from an edge by attaching two leaves to both endpoints. For $T \in \{T_{7,10},T_{7,11}\}$ let $T'=T_6$. Note that in $T_6$ every vertex is adjacent to a vertex of degree three. Using this fact, it is easy to check that for all vertices $v \in T'$ and $w \in T$, $\left(T' \cup kT \right) + v\mathbf{w}$ contains a copy of $T$ in which the copy of the circled vertex is $v$, where $k=3$ if $T=T_{7,10}$ and $k=2$ if $T=T_{7,11}$.

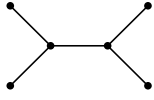
\begin{figure}[H]
\centering
\scalebox{1.5}{
\begin{tikzpicture}

    \filldraw[black] (-1.5+0.3535+2,1.3535-6.5) circle (0.75pt);
    \filldraw[black] (-1.5+0.3535+2,1-0.3535-6.5) circle (0.75pt);
    \filldraw[black] (-1.1464+0.3535+2,1-6.5) circle (0.75pt);
    \filldraw[black] (-0.6464+0.3535+2,1-6.5) circle (0.75pt);
    \filldraw[black] (-0.6464+0.3535+0.3535+2,1-6.5+0.3535) circle (0.75pt);
    \filldraw[black] (-0.6464+0.3535+0.3535+2,1-6.5-0.3535) circle (0.75pt);
    \draw  (-1.5+0.3535+2,1.3535-6.5) -- (-1.1464+0.3535+2,1-6.5);
    \draw  (-1.5+0.3535+2,1-0.3535-6.5) -- (-1.1464+0.3535+2,1-6.5);
    \draw  (-1.1464+0.3535+2,1-6.5) -- (-0.6464+0.3535+2,1-6.5);
    \draw   (-0.6464+0.3535+2,1-6.5) --  (-0.6464+0.3535+0.3535+2,1-6.5+0.3535);
    \draw   (-0.6464+0.3535+2,1-6.5) -- (-0.6464+0.3535+0.3535+2,1-6.5-0.3535);
   
\end{tikzpicture}}
\caption{The tree $T_6$.}\label{fig:expl}
\end{figure}

\end{proof}

We are now ready to prove Theorem \ref{quasfinfam}.

\begin{repth2}\leavevmode

\begin{enumerate}

    \item Let $\mathcal{F}$ be a quasi-finite family of graphs. Then for each integer $r \geq 1$, there is a set $S_r$ of residue classes modulo $r$ and for each residue class $C \in S_r$, there is an integer $k_C$, such that the following hold.

        \begin{enumerate}
        
        \item For every integer $r \geq 1$ and residue class $C \in S_r$,
        
        $$ sat(n,\mathcal{F})= \frac{(r-1)n+k_C}{r} $$
        
        for large enough $n \in C \setminus \bigcup_{q<r, D \in S_q} D$.
        
        \item We have
        
        $$ sat(n,\mathcal{F}) \geq n+o(n) $$
        
        for $n \in \mathbb{Z}_{\geq 0} \setminus \bigcup_{r \geq 1, C \in S_r} C$.

        \item For $r \in \{1,2,3,4,5,7\}$, $S_r$ contains either none or all of the residue classes modulo $r$.
        
    \end{enumerate}

    \item For each integer $r \geq 1$, let $S_r$ be a set of residue classes modulo $r$. Suppose that for $r \in \{1,2,3,4,5,7\}$, $S_r$ contains either none or all of the residue classes modulo $r$. Then there exists a quasi-finite family $\mathcal{F}$ of graphs such that the following hold. 

     \begin{enumerate}
     
        \item For every integer $r \geq 1$, 
        
        $$sat(n,\mathcal{F})= \left(1-\frac{1}{r}\right)n+O(1)$$
        
        for $n \in (\bigcup_{C \in S_r} C) \setminus \bigcup_{q<r, C \in S_q} C$.
        
        \item For $n \in \mathbb{Z}_{\geq 0} \setminus \bigcup_{r \geq 1, C \in S_r} C$,

        $$ sat(n,\mathcal{F}) \geq n+o(n) \ . $$ 
        
        \item If there are only finitely many integers $r \geq 1$ with $(\bigcup_{C \in S_r} C) \setminus \bigcup_{q<r, C \in S_q} C \neq \emptyset$, $\mathcal{F}$ is in fact finite.
        
    \end{enumerate}

\end{enumerate}

\end{repth2}

\begin{proof}

We first prove part 1. Let $\mathcal{F}$ be a quasi-finite family of graphs. Let us call a pair $(H,T)$, where $H$ is a graph and $T$ is a tree, a \emph{blueprint}. \newline

\emph{Claim 1:} For every tree $T$ there is a constant $C_T \geq 0$ such that for every blueprint $(H,T)$, either $H \cup kT$ is $\mathcal{F}$-saturated for every integer $k \geq C_T$ or $H \cup kT$ is not $\mathcal{F}$-saturated for every integer $k \geq C_T$. \newline

\begin{figure}[H]
\centering
\scalebox{1}{
\begin{tikzpicture}

    \draw (0,0) ellipse (0.75 and 0.75);
    \node[anchor=center] at (0,0) {$H$};

    \begin{scope}[shift={(-2.25/2,0)}]
    \draw (4.5,-3) ellipse (0.75 and 0.75);
    \node[anchor=center] at (4.5,-3) {$T$};
    \end{scope}
    
    \begin{scope}[shift={(2.25/2,0)}]
    \draw (-2.25,-3) ellipse (0.75 and 0.75);
    \node[anchor=center] at (-2.25,-3) {$T$};

    \draw (-4.5,-3) ellipse (0.75 and 0.75);
    \node[anchor=center] at (-4.5,-3) {$T$};
    \end{scope}
    
    \filldraw[fill=black!0] (2.25/2,-3) circle (0cm) node[black]{$\cdots$};

    \draw [decorate,decoration={brace,amplitude=10pt,raise=5pt}] (2.25*1.5+0.75,-3.75) -- (-2.25*1.5-0.75,-3.75) node[midway,below=20pt] {$k$};

\end{tikzpicture}}
\caption{The graph $H \cup kT$.}\label{fig:HkT}
\end{figure}
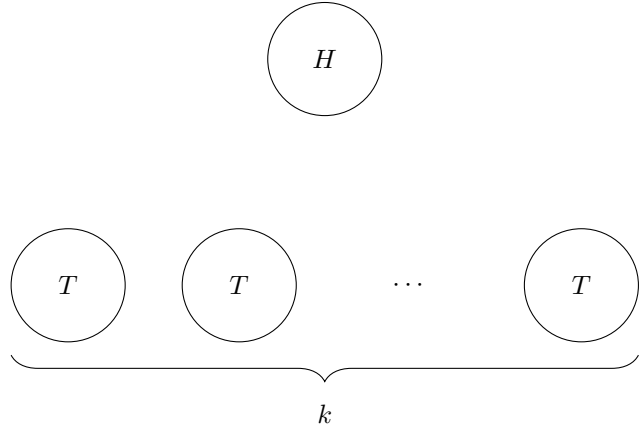

\emph{Proof:} Let $T$ be a tree. Since $\mathcal{F}$ is quasi-finite, for every subtree $T' \subseteq T$ there is a constant $C'_{T'} \geq 0$ such that every graph in $\mathcal{F}$ has at most $C'_{T'}$ connected components isomorphic to $T'$. We first show that for every blueprint $(H,T)$ and integers $k \geq 0$ and $\ell \geq \sum_{T' \subseteq T} C'_{T'}$, if $H \cup kT$ contains a graph $F \in \mathcal{F}$ then so does $H \cup \ell T$. \newline

Consider a copy of $F$ in $H \cup kT$. Every connected component of $F$ must be contained in either $H$ or one of the copies of $T$. Since $T$ is acyclic, every connected component of $F$ contained in a copy of $T$ is isomorphic to a subtree $T' \subseteq T$, so there are at most $\sum_{T' \subseteq T} C'_{T'} \leq \ell$ such connected components. Hence there are at most $\ell$ copies of $T$ containing a connected component of $F$, so $F$ is contained in $H \cup \ell T$. \newline

Now let $C_T=\sum_{T' \subseteq T} C'_{T'} + 2$. Let $(H,T)$ be a blueprint. We need to show that if $H \cup kT$ is $\mathcal{F}$-saturated for some integer $k \geq C_T$, then in fact $H \cup \ell T$ is $\mathcal{F}$-saturated for every integer $\ell \geq C_T$. $H \cup \ell T$ must be $\mathcal{F}$-free, for otherwise $H \cup kT$ would contain a graph in $\mathcal{F}$, since $k \geq C_T \geq \sum_{T' \subseteq T} C'_{T'}$. It remains to show that every graph $\left(H \cup \ell T\right)+vw$ obtained by adding a new edge $vw$ to $H \cup \ell T$ contains a graph in $\mathcal{F}$. \newline

Suppose first that $v,w \in H$. Then $\left(H \cup \ell T\right)+vw=\left(H + vw \right) \cup \ell T$ contains a graph in $\mathcal{F}$, since $\left(H + vw \right) \cup kT=\left(H \cup kT\right)+vw$ contains a graph in $\mathcal{F}$ and $\ell \geq C_T \geq \sum_{T' \subseteq T} C'_{T'}$. Now suppose $v \in H$ and $w$ is in a copy $T_1$ of $T$. Let $w' \in H \cup kT$ be the copy of $w$ in a copy $T_2$ of $T$ (note $k \geq C_T \geq 1$). Then $\left(H \cup \ell T\right)+vw=\left[\left(H \cup T_1\right) + vw \right] \cup (\ell-1)T$ contains a graph in $\mathcal{F}$, since $\left[\left(H \cup T_2\right) + vw' \right] \cup (k-1)T = \left(H \cup kT\right)+vw'$ contains a graph in $\mathcal{F}$ and $\ell-1 \geq C_T-1 \geq \sum_{T' \subseteq T} C'_{T'}$. \newline

Next, suppose $v$ and $w$ are in the same copy $T_1$ of $T$. Let $v', w' \in H \cup kT$ be the copies of $v$ and $w$ in a copy $T_2$ of $T$ (again note $k \geq C_T \geq 1$). Then $\left(H \cup \ell T\right)+vw= H \cup \left(T_1+vw\right) \cup (\ell-1)T$ contains a graph in $\mathcal{F}$, since $H \cup \left(T_2+v'w'\right) \cup (k-1)T = \left(H \cup kT\right)+v'w'$ contains a graph in $\mathcal{F}$ and $\ell-1 \geq C_T-1 \geq \sum_{T' \subseteq T} C'_{T'}$. Finally, suppose $v$ and $w$ are in different copies $T_1$ and $T_2$ of $T$. Let $v',w' \in H \cup kT$ be the copies of $v$ and $w$ in different copies $T_3$ and $T_4$ of $T$ (note $k \geq C_T \geq 2$). Then $\left(H \cup \ell T\right)+vw = H \cup \left[\left(T_1 \cup T_2 \right)+vw \right] \cup ( \ell -2)T$ contains a graph in $\mathcal{F}$, since $H \cup \left[\left(T_3 \cup T_4 \right)+v'w' \right] \cup (k-2)T=\left(H \cup kT \right)+v'w'$ contains a graph in $\mathcal{F}$ and $\ell-2 \geq C_T-2 = \sum_{T' \subseteq T} C'_{T'}$. \qedclaim\newline

Let us call a blueprint $(H,T)$ $\mathcal{F}$-saturated if $H \cup kT$ is $\mathcal{F}$-saturated for every large enough integer $k$. \newline

\emph{Claim 2:} For every integer $r \geq 1$ there is a constant $C_r$ such that every $\mathcal{F}$-saturated graph $G$ on $n$ vertices that is not of the form $H \cup kT$ for some $\mathcal{F}$-saturated blueprint $(H,T)$ with $|T| < r$ and integer $k \geq 0$ has at least $\left(1-1/r \right)n -C_r $ edges. \newline

\emph{Proof:} For each integer $r \geq 1$, let $C_r=\sum_{|T|<r} C_T$ (where the $C_T$ are as in Claim 1). Let $G$ be as in Claim 2. Then for every tree $T$ with $|T|<r$, $G \cong H \cup kT$, where $H$ is the union of all connected components of $G$ that are not isomorphic to $T$ and $k$ is the number of connected components of $G$ that are isomorphic to $T$. Hence the blueprint $(H,T)$ is not $\mathcal{F}$-saturated, so $k<C_T$ by Claim 1.  It follows that $t(G) \leq n/r + C_r$, where $t(G)$ is the number of connected components of $G$ that are trees. For every connected component $C$, $e(C) \geq |C|-1$, with equality if and only if $C$ is a tree. Summing over all $C$, we obtain $e(G) \geq n-t(G) \geq \left(1-1/r \right)n -C_r$. \qedclaim\newline

For each integer $r \geq 1$, let $S_r$ be the set of residue classes $C$ modulo $r$ for which there exists an $\mathcal{F}$-saturated blueprint $(H,T)$ with $|H| \in C$ and $|T|=r$. Note that $|H \cup kT| \equiv |H| \mod |T| $ for every blueprint $(H,T)$ and integer $k \geq 0$. Hence, by Claim 2, $sat(n,\mathcal{F}) \geq \left(1-1/r \right)n - C_r $ for every integer $r \geq 1$ and $n \in \mathbb{Z}_{\geq 0} \setminus \bigcup_{q<r,C \in S_q} C $. Since $r$ is arbitrary, this proves part (b). \newline

We now prove part (a). Let $r \geq 1$ be an integer and $C \in S_r$ be a residue class modulo $r$. If $C \setminus \bigcup_{q<r, D \in S_q} D = \emptyset$, part (a) is vacuously true (so let $k_C$ be any integer). Otherwise, let $k_C$ be the minimum of $re(H)-(r-1)|H|$ over all $\mathcal{F}$-saturated blueprints $(H,T)$ with $|H| \in C$ and $|T|=r$. We need to show that this quantity is bounded below, so that the minimum is well-defined. \newline

Let $(H,T)$ be an $\mathcal{F}$-saturated blueprint with $|H| \in C$ and $|T|=r$. Since $C \setminus \bigcup_{q<r, D \in S_q} D$ is non-empty and a union of residue classes modulo the lowest common multiple of the integers $1 \leq q \leq r$, it contains arbitrarily large integers. So pick an integer $n \in C \setminus \bigcup_{q<r, D \in S_q} D$ with $n \geq |H|+ rC_T$ (where $C_T$ is as in Claim 1). Then $n=|H|+kr$ for some integer $k \geq C_T$. But then $n \in \mathbb{Z}_{\geq 0} \setminus \bigcup_{q<r,D \in S_q} D $ and by Claim 1, $H \cup kT$ is $\mathcal{F}$-saturated, so 

$$\left(1-\frac{1}{r} \right)n -C_r \leq sat(n,\mathcal{F}) \leq e\left(H \cup kT \right) = \left(1-\frac{1}{r} \right)n + \frac{re(H)-(r-1)|H|}{r} \ , $$

so $re(H)-(r-1)|H| \geq -rC_r$. \newline

Now let $(H,T)$ be an $\mathcal{F}$-saturated blueprint with $|H| \in C$ and $|T|=r$ that minimises $re(H)-(r-1)|H|$, i.e. with $re(H)-(r-1)|H|=k_C$. Then by Claim 1, $H \cup kT$ is $\mathcal{F}$-saturated for every integer $k \geq C_T$, so there exist $\mathcal{F}$-saturated graphs with $n$ vertices and $\left[(r-1)n+k_C \right]/r$ edges for every integer $|H|+rC_T \leq n \in C$. \newline

It remains to show that every $\mathcal{F}$-saturated graph $G$ on $n$ vertices has at least $\left[(r-1)n+k_C \right]/r$ edges for large enough $n \in C \setminus \bigcup_{q<r, D \in S_q} D$. If $G$ is not of the form $H \cup kT$ for some $\mathcal{F}$-saturated blueprint $(H,T)$ with $|T|=r$ and integer $k \geq 0$, then by Claim 2, 

$$ e(G) \geq \left(1- \frac{1}{r+1} \right) n - C_{r+1} > \frac{(r-1)n+k_C }{r} $$

if $n$ is large enough. If $G$ is of the form $H \cup kT$ for some $\mathcal{F}$-saturated blueprint $(H,T)$ with $|T|=r$ and integer $k \geq 0$, 

$$ e(G) = \frac{(r-1)n+ \left[ re(H)-(r-1)|H| \right] }{r}  \geq \frac{(r-1)n+k_C }{r} $$ 

by the definition of $k_C$. \newline

Next, we prove part (c). \newline

\emph{Claim 3:} Let $(H,T)$ be an $\mathcal{F}$-saturated blueprint and $T' \subseteq T$ be a subtree such that for all vertices $v \in T'$ and $w \in T$, $\left(T' \cup kT \right) + v\mathbf{w}$ contains a copy of $T$ that contains $v$ for some integer $k \geq 0$ (where $\left(T' \cup kT \right) + v\mathbf{w}$ is as in Lemma \ref{expl}). Then there exists an $\mathcal{F}$-saturated blueprint $(H',T)$ with $|H'| \equiv |H|+|T'| \mod |T|$. \newline

\emph{Proof:} For each $v \in T'$ and $w \in T$, let $k_{v,w} \geq 0$ be an integer such that $\left(T' \cup k_{v,w}T \right) + v\mathbf{w}$ contains a copy of $T$ that contains $v$. Let $k \geq C_T + \sum_{v \in T', w \in T} k_{v,w}$ be an integer (where $C_T$ is as in Claim 1). Then $H \cup kT \cup T' \subseteq H \cup (k+1)T$ and $ H \cup (k+1)T$ is $\mathcal{F}$-free by Claim 1, since $k+1 \geq C_T$, so $H \cup kT \cup T'$ is $\mathcal{F}$-free. Hence we can add edges to $H \cup kT \cup T'$ to obtain an $\mathcal{F}$-saturated graph $G$. $H \cup kT$ is $\mathcal{F}$-saturated by Claim 1, since $k \geq C_T$, so all new edges in $G$ are incident to $T'$. \newline

Let $v \in T'$ and $w \in T$. Suppose for the sake of contradiction that $v$ is adjacent to more than $k_{v,w}$ copies of $w$ in $G$. Then $G \supseteq H \cup \left( \left[T' \cup \left(k_{v,w}+1\right)T \right] + v\mathbf{w} \right) \cup (k-k_{v,w}-1) T$. But $\left(T' \cup k_{v,w}T \right) + v\mathbf{w}$ contains a copy of $T$ that contains $v$, so $ \left[T' \cup (k_{v,w}+1)T \right] + v\mathbf{w} \supseteq 2T+e$ for some edge $e$ between the two copies of $T$ in $2T$. Hence $ G \supseteq H \cup \left(2T+e\right) \cup (k-k_{v,w}-1) T = \left[H \cup (k-k_{v,w}+1)T \right] + e$. But $\left[H \cup (k-k_{v,w}+1)T \right] + e$ contains a graph in $\mathcal{F}$ by Claim 1, since $k-k_{v,w}+1 \geq C_T$, so $G$ contains a graph in $\mathcal{F}$, a contradiction. \newline

Hence, for all vertices $v \in T'$ and $w \in T$, $v$ is adjacent to at most $k_{v,w}$ copies of $w$ in $G$, so there are at most $\sum_{v \in T', w \in T} k_{v,w}$ copies of $T$ in $G$ with some edge between them and $T'$. Let $H'$ be the subgraph of $G$ induced by $H$, $T'$ and these copies of $T$. Note that $|H'| \equiv |H|+|T'| \mod |T|$. Then $G=H' \cup \ell T$ for some integer $\ell \geq k-\sum_{v \in T', w \in T} k_{v,w} \geq C_T$, so $(H',T)$ is an $\mathcal{F}$-saturated blueprint by Claim 1. \qedclaim\newline

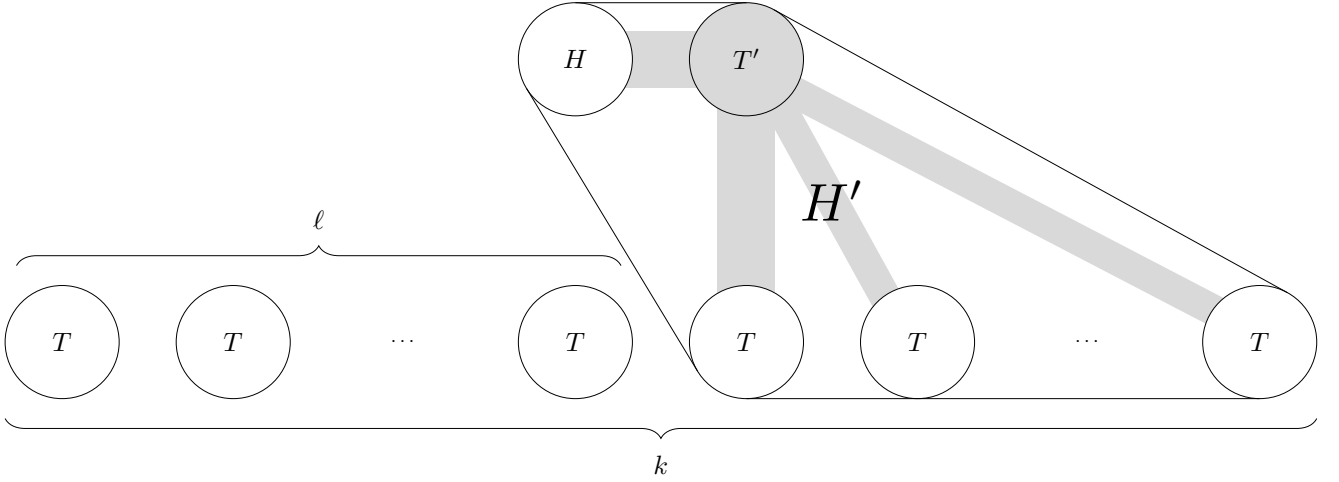
\begin{figure}[H]
\centering
\scalebox{0.75}{
\begin{tikzpicture}

%\fill[black!15] (0.5*3^0.5,-0.5) arc (-30:30:1) --(2.5,0.5)--(2.5,-0.5)--(0.5*3^0.5,-0.5); 
\fill[black!15] (0.5,0.5) --(2.5,0.5)--(2.5,-0.5)--(0.5,-0.5); 
\fill[black!15] (2.5,-0.5) --(2.5,-4.5)--(3.5,-4.5)--(3.5,-0.5); 
\fill[black!15] (0.5,0.5) --(2.5,0.5)--(2.5,-0.5)--(0.5,-0.5); 
\fill[black!15] ($ (6,-5)+(100:0.5) $) -- ($ (6,-5)+(170:0.5) $) -- ($ (3,0)+(-80:0.5) $) -- ($ (3,0)+(-10:0.5) $); 
\fill[black!15] ($ (12,-5)+(100:0.5) $) -- ($ (12,-5)+(170:0.5) $) -- ($ (3,0)+(-80:0.5) $) -- ($ (3,0)+(-10:0.5) $);

\filldraw[fill=black!0] (0,0) circle (1);
\node[anchor=center] at (0,0) {\scalebox{1.3333}{$H$}};
\filldraw[fill=black!15] (3,0) circle (1);
\node[anchor=center] at (3,0) {\scalebox{1.3333}{$T'$}};
\draw (-9,-5) circle (1);
\node[anchor=center] at (-9,-5) {\scalebox{1.3333}{$T$}};
\draw (-6,-5) circle (1);
\node[anchor=center] at (-6,-5) {\scalebox{1.3333}{$T$}};
\draw (0,-5) circle (1);
\node[anchor=center] at (0,-5) {\scalebox{1.3333}{$T$}};
\filldraw[fill=black!0] (3,-5) circle (1);
\node[anchor=center] at (3,-5) {\scalebox{1.3333}{$T$}};
\filldraw[fill=black!0] (6,-5) circle (1);
\node[anchor=center] at (6,-5) {\scalebox{1.3333}{$T$}};
\filldraw[fill=black!0] (12,-5) circle (1);    
\node[anchor=center] at (12,-5) {\scalebox{1.3333}{$T$}};
\node[anchor=center] at (-3,-5) {$\cdots$};
\node[anchor=center] at (9,-5) {$\cdots$};
\draw [decorate,decoration={brace,amplitude=10pt,raise=10pt}] (13,-6) -- (-10,-6) node[midway,below=25pt] {\scalebox{1.3333}{$k$}};
\draw [decorate,decoration={brace,amplitude=10pt,raise=10pt}] (-9.8,-4) -- (0.8,-4) node[midway,above=25pt] {\scalebox{1.3333}{$\ell$}};
\draw (210.963756532:1) -- ($ (3,-5)+(210.963756532:1) $);
\draw (0,1) -- (3,1);
\draw (3,-6) -- (12,-6);
\draw ($ (3,0)+(60.9454:1) $) -- ($ (12,-5)+(60.9454:1) $);
\node[anchor=center] at (4.5,-2.5) {\scalebox{2.66666}{$H'$}};

\end{tikzpicture}}
\caption{The graph $G$ and its induced subgraph $H'$. The shading denotes the presence of new edges.}\label{fig:GH'}
\end{figure}

For all positive, coprime integers $a$ and $b$, every integer is congruent to a positive multiple of $a$ modulo $b$. Hence part (c) follows from Lemma \ref{expl} by repeatedly applying Claim 3. \newline

We now prove part 2. We first define, for each integer $r \geq 1$, a tree $T_r$ and a family of graphs $\mathcal{F}_r$ with the following properties.

\begin{enumerate}
    \item $|T_r|=r$.
    \item $\mathcal{F}_r$ is finite.
    \item Every graph in $\mathcal{F}_r$ is connected.
    \item $T_r$ is $\mathcal{F}_r$-free.
    \item Every graph $T_r+e$ obtained by adding a new edge $e$ to $T_r$ and every graph $2T_r+e$ obtained from $2T_r$ by adding an edge $e$ between the two copies of $T_r$ contains a graph in $\mathcal{F}_r$.
\end{enumerate}

For each $r \in \{1,2,3,4,5,7\}$, let $T_r$ be any tree with $|T_r|=r$ and $\mathcal{F}_r$ be the family containing every graph $T_r+e$ obtained by adding a new edge $e$ to $T_r$ and every graph  $2T_r+e$ obtained from $2T_r$ by adding an edge $e$ between the two copies of $T_r$. Note that properties 1 through 5 hold for every $r \in \{1,2,3,4,5,7\}$. \newline

For each even integer $r \geq 6$, let $T_r$ be the tree obtained from an edge by attaching $r/2-1$ leaves to both endpoints and $\mathcal{F}_r'$ be the family of graphs in which every vertex is adjacent to a vertex of degree at least $r/2$. For each odd integer $r \geq 9$, let $T_r$ be the tree obtained from a three vertex path by attaching two leaves to one of the endpoints, one leaf to the central vertex and $r-6$ leaves to the other endpoint and $\mathcal{F}_r'$ be the family of graphs in which every vertex $v$ is either adjacent to a vertex of degree at least $r-5$ or adjacent to a vertex of degree at least three which is adjacent to a vertex distinct from $v$ of degree at least three. Note that property 1 holds for every $r \in \{6\} \cup \mathbb{Z}_{\geq 8}$. \newline

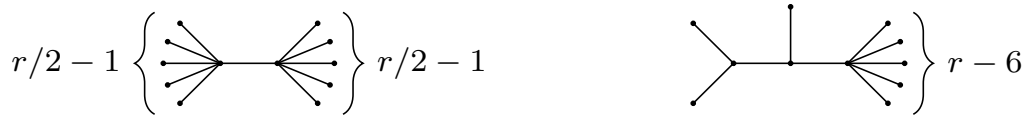
\begin{figure}[H]
\centering
\scalebox{1.5}{
\begin{tikzpicture}

    \filldraw[black] (-1.1464+0.3535+2,1-6.5) circle (0.5pt);
    \filldraw[black] (-1.1464+0.3535+2,1-6.5)+(180+45/2:0.5) circle (0.5pt);
    \filldraw[black] (-1.1464+0.3535+2,1-6.5)+(180-45/2:0.5) circle (0.5pt);
    \filldraw[black] (-1.1464+0.3535+2,1-6.5)+(180+45:0.5) circle (0.5pt);
    \filldraw[black] (-1.1464+0.3535+2,1-6.5)+(180-45:0.5) circle (0.5pt);
    \filldraw[black] (-1.1464+0.3535+2,1-6.5)+(180:0.5) circle (0.5pt);

    \draw  (-1.1464+0.3535+2,1-6.5) -- ($ (-1.1464+0.3535+2,1-6.5)+(180+45/2:0.5) $);
    \draw   (-1.1464+0.3535+2,1-6.5) --  ($ (-1.1464+0.3535+2,1-6.5)+(180-45/2:0.5) $);
    \draw   (-1.1464+0.3535+2,1-6.5) -- ($ (-1.1464+0.3535+2,1-6.5)+(180+45:0.5)  $);
    \draw   (-1.1464+0.3535+2,1-6.5) -- ($ (-1.1464+0.3535+2,1-6.5)+(180-45:0.5) $);
    \draw (-1.1464+0.3535+2,1-6.5) --   ($ (-1.1464+0.3535+2,1-6.5)+(180:0.5) $);

    \filldraw[black] (-0.6464+0.3535+2,1-6.5)+(45/2:0.5) circle (0.5pt);
    \filldraw[black] (-0.6464+0.3535+2,1-6.5)+(-45/2:0.5) circle (0.5pt);
    \filldraw[black] (-1.1464+0.3535+2+1,1-6.5) circle (0.5pt);
    \filldraw[black] (-0.6464+0.3535+2,1-6.5) circle (0.5pt);
    \filldraw[black] (-0.6464+0.3535+0.3535+2,1-6.5+0.3535) circle (0.5pt);
    \filldraw[black] (-0.6464+0.3535+0.3535+2,1-6.5-0.3535) circle (0.5pt);

    \draw  (-1.1464+0.3535+2,1-6.5) -- (-0.6464+0.3535+2,1-6.5);
    \draw   (-0.6464+0.3535+2,1-6.5) --  (-0.6464+0.3535+0.3535+2,1-6.5+0.3535);
    \draw   (-0.6464+0.3535+2,1-6.5) -- (-0.6464+0.3535+0.3535+2,1-6.5-0.3535);
    \draw   (-1.1464+0.3535+2+1,1-6.5) -- (-1.1464+0.3535+2+0.5,1-6.5);
    \draw (-0.6464+0.3535+2,1-6.5)+(45/2:0.5) --  (-0.6464+0.3535+2,1-6.5);
    \draw (-0.6464+0.3535+2,1-6.5)+(-45/2:0.5) --  (-0.6464+0.3535+2,1-6.5);

    \draw [decorate,decoration={brace,amplitude=5pt,raise=5pt}] (-1.1464+0.3535+2-0.4,1-6.5-0.45) -- (-1.1464+0.3535+2-0.4,1-6.5+0.45) node[midway,left=10pt] {\footnotesize $r/2-1$};

     \draw [decorate,decoration={brace,amplitude=5pt,raise=5pt}] (-1.1464+0.3535+2+0.4+0.5,1-6.5+0.45) -- (-1.1464+0.3535+2+0.4+0.5,1-6.5-0.45) node[midway,right=10pt] {\footnotesize $r/2-1$};

    \begin{scope}[shift={(5,0)}]

        \begin{scope}[shift={(-0.5,0)}]
    
        \filldraw[black] (-1.5+0.3535+2,1.3535-6.5) circle (0.5pt);
        \filldraw[black] (-1.5+0.3535+2,1-0.3535-6.5) circle (0.5pt);
        \filldraw[black] (-1.1464+0.3535+2,1-6.5) circle (0.5pt);
        \draw  (-1.5+0.3535+2,1.3535-6.5) -- (-1.1464+0.3535+2,1-6.5);
        \draw  (-1.5+0.3535+2,1-0.3535-6.5) -- (-1.1464+0.3535+2,1-6.5);
        \draw  (-1.1464+0.3535+2,1-6.5) -- (-0.6464+0.3535+2,1-6.5);
        \filldraw[black] (-1.1464+0.3535+2+0.5,1-6.5) circle (0.5pt);
        \end{scope}
    
    \filldraw[black] (-1.1464+0.3535+2+1,1-6.5) circle (0.5pt);
    \filldraw[black] (-0.6464+0.3535+2,1-6.5) circle (0.5pt);
    \filldraw[black] (-0.6464+0.3535+0.3535+2,1-6.5+0.3535) circle (0.5pt);
    \filldraw[black] (-0.6464+0.3535+0.3535+2,1-6.5-0.3535) circle (0.5pt);
    \filldraw[black]  (-1.1464+0.3535+2,1-6) circle (0.5pt);
    \filldraw[black] (-0.6464+0.3535+2,1-6.5)+(45/2:0.5) circle (0.5pt);
    \filldraw[black] (-0.6464+0.3535+2,1-6.5)+(-45/2:0.5) circle (0.5pt);
    \draw   (-0.6464+0.3535+2,1-6.5) --  (-0.6464+0.3535+0.3535+2,1-6.5+0.3535);
    \draw   (-0.6464+0.3535+2,1-6.5) -- (-0.6464+0.3535+0.3535+2,1-6.5-0.3535);
    \draw   (-1.1464+0.3535+2+1,1-6.5) -- (-1.1464+0.3535+2+0.5,1-6.5);
    \draw   (-1.1464+0.3535+2,1-6.5) --  (-1.1464+0.3535+2+0.5,1-6.5); 
    \draw   (-1.1464+0.3535+2,1-6.5) --  (-1.1464+0.3535+2,1-6); 
    \draw (-0.6464+0.3535+2,1-6.5)+(45/2:0.5) --  (-0.6464+0.3535+2,1-6.5);
    \draw (-0.6464+0.3535+2,1-6.5)+(-45/2:0.5) --  (-0.6464+0.3535+2,1-6.5);
    \draw [decorate,decoration={brace,amplitude=5pt,raise=5pt}] (-1.1464+0.3535+2+1-0.1,1-6.5+0.45)  -- (-1.1464+0.3535+2+1-0.1,1-6.5-0.45)  node[midway,right=10pt] {\footnotesize $r-6$};

    \end{scope}

\end{tikzpicture}}
\caption{The tree $T_r$. The cases $r \geq 6$ even and $r \geq 9$ odd are on the left and right, respectively.}\label{fig:expl}
\end{figure}

Let $r \in \{6\} \cup \mathbb{Z}_{\geq 8}$. Let $\mathcal{F}_r$ be the family of minimal connected graphs in $\mathcal{F}_r' \setminus \{T_r\}$, that is, connected graphs in $\mathcal{F}_r' \setminus \{T_r\}$ that do not contain a smaller connected graph in $\mathcal{F}_r' \setminus \{T_r\}$. Then clearly property 3 holds. A graph is in $\mathcal{F}_r'$ if and only if all of its connected components are in $\mathcal{F}_r'$, every graph obtained by adding edges to a graph in $\mathcal{F}_r'$ is also in $\mathcal{F}_r'$ and one can check that the subgraphs of $T_r$ in $\mathcal{F}_r'$ are precisely the subgraph with no vertices and $T_r$ itself, which together imply properties 4 and 5. It remains to show property 2 holds. \newline

 \emph{Claim 4:} For all $r \in \{6\} \cup \mathbb{Z}_{\geq 8}$ and $v \in G \in \mathcal{F}_r'$ there is a set $S$ of vertices in $G$ with $v \in S$, $G[S]$ connected, $G[S] \in \mathcal{F}_r'$ and $|S| \leq 2r$. \newline

\emph{Proof:} Suppose first that $r \geq 6$ is even. Then $v$ is adjacent to a vertex $a$ of degree at least $r/2$ and $a$ is adjacent to a vertex $b$ of degree at least $r/2$. Pick sets $\{v,b\} \subseteq A \subseteq \Gamma(a)$ and $\{a\} \subseteq B \subseteq \Gamma(b)$ of size $r/2$. Then $S=A \cup B$ is the desired set. Now suppose $r \geq 9$ is odd. Suppose first that $v$ is adjacent to a vertex $a$ of degree at least $r-5$. If $a$ is adjacent to a vertex $b$ of degree at least $r-5$, pick sets $\{v,b\} \subseteq A \subseteq \Gamma(a)$ and $\{a\} \subseteq B \subseteq \Gamma(b)$ of size $r-5$. Then $S=A \cup B$ is the desired set. If $a$ is adjacent to a vertex $b$ of degree at least three which is adjacent to a vertex $c \neq a$ of degree at least three, pick a set $\{v,b\} \subseteq A \subseteq \Gamma(a)$ of size $r-5$ and sets $\{a,c\} \subseteq B \subseteq \Gamma(b)$ and $C \subseteq \Gamma(c)$ of size three. Then $S=A \cup B \cup C$ is the desired set. \newline

Now suppose $v$ is adjacent to a vertex $a$ of degree at least three which is adjacent to a vertex $b$ of degree at least three.  Let $w \in \{a,b\}$. If $w$ is adjacent to a vertex $c_w$ of degree at least $r-5$, pick a set $S_w \subseteq \Gamma(c_w)$ of size $r-5$. If $w$ is adjacent to a vertex $c_w$ of degree at least three which is adjacent to a vertex $d_w \neq w$ of degree at least three, pick sets $\{d_w\} \subseteq C_w \subseteq \Gamma(c_w)$ and $D_w \subseteq \Gamma(d_w)$ of size three and let $S_w=C_w \cup D_w$. Now pick sets $\{v,b,c_a\} \subseteq A \subseteq \Gamma(a)$ and $\{a,c_b\} \subseteq B \subseteq \Gamma(b)$ of size three. Then $S=A \cup B \cup S_a \cup S_b$ is the desired set.  \qedclaim\newline

\emph{Claim 5:} For all $r \in \{6\} \cup \mathbb{Z}_{\geq 8}$, every graph in $\mathcal{F}_r$ has at most $2r$ vertices. In particular, $\mathcal{F}_r$ is finite. \newline

\emph{Proof:} Let $G \in \mathcal{F}_r$. By Claim 4, for every vertex $v \in G$ there is a set $S_v$ of vertices in $G$ with $v \in S_v$, $G[S_v]$ connected, $G[S_v] \in \mathcal{F}_r'$ and $|S_v| \leq 2r$. If for some vertex $v$, $G[S_v]$ is not isomorphic to $T_r$, by the minimality of $G$, $S_v$ is the entire vertex set of $G$ and hence $|G| \leq 2r$. So suppose $G[S_v]$ is isomorphic to $T_r$ for every vertex $v$. If for some vertices $v$ and $w$, $S_v$ and $S_w$ are distinct and intersect, $G[S_v \cup S_w]$ is a connected graph in $\mathcal{F}_r' \setminus \{T_r\}$, so by the minimality of $G$, $S_v \cup S_w$ is the entire vertex set of $G$ and hence $|G| < 2r$ by property 1. So suppose the $S_v$ partition the vertex set of $G$. If there is an edge between two different parts $S_v$ and $S_w$, $G[S_v \cup S_w]$ is a connected graph in $\mathcal{F}_r' \setminus \{T_r\}$, so by the minimality of $G$, $S_v \cup S_w$ is the entire vertex set of $G$ and hence $|G|=2r$ by property 1. So suppose $G$ is a disjoint union of copies of $T_r$. This is impossible, since $G$ is connected and not isomorphic to $T_r$. \qedclaim\newline

For each integer $r \geq 1$, let $S_r$ be a set of residue classes modulo $r$. Suppose that for $r \in \{1,2,3,4,5,7\}$, $S_r$ contains either none or all of the residue classes modulo $r$. We need to construct a quasi-finite family $\mathcal{F}$ of graphs such that (a), (b) and (c) hold. For each integer $r \geq 1$, let $S_r'$ be the set of residue classes $C \in S_r$ with $C \setminus \bigcup_{q<r, D \in S_q} D \neq \emptyset$. Pick an integer $3 \leq k_C \in C$ for each integer $r \geq 1$ and residue class $C \in S_r'$ such that the $k_C$ are distinct. This is possible, since if we first pick $k_C$ for each $C \in S_1'$, then pick $k_C$ for each $C \in S_2'$, then pick $k_C$ for each $C \in S_3'$ and so on, then at each step we can let $k_C$ be any large enough integer in $C$. For each integer $k \geq 3$, let $C_k$ be the cycle on $k$ vertices. Let $\mathcal{C}$ be the family containing $C_{k_C}$ for every integer $r \geq 1$ and residue class $C \in S_r'$. For each integer $r \geq 1$ and residue class $C \in S_r'$, let $\mathcal{F}_{C}$ be the family containing the following graphs.

\begin{enumerate}
    \item $C_{k_C} \cup H$ for every graph $H \in \mathcal{F}_r$.
    \item Every graph $C_{k_C}+e$ obtained by adding a new edge $e$ to $C_{k_C}$.
    \item The graph obtained from $C_{k_C}$ by adding a new vertex and joining it to a vertex in  $C_{k_C}$.
    \item  $C_{k_C} \cup H$ for every graph $H \in \mathcal{C}$.
\end{enumerate}

\begin{figure}[H]
\centering
\begin{tikzpicture}

    \draw (-2.5,0) ellipse (0.75 and 0.75);
    \node[anchor=center] at (-2.5,0) {$k_C$};
    \filldraw[black] (-2.5,0.75) circle (0.75pt);
    \filldraw[black] (-2.5,0)+(90+360/5:0.75) circle (0.75pt);
    \filldraw[black] (-2.5,0)+(90+2*360/5:0.75) circle (0.75pt);
    \filldraw[black] (-2.5,0)+(90+3*360/5:0.75) circle (0.75pt);
    \filldraw[black] (-2.5,0)+(90+4*360/5:0.75) circle (0.75pt);
    \draw (-0.5,0) ellipse (0.75 and 0.75);
    \node[anchor=center] at (-0.5,0) {$H$};
    \node[anchor=north] at (-1.5,-1.5) {Type 1};

    \begin{scope}[shift={(6.5,0)}]
    \draw (-2.5,0) ellipse (0.75 and 0.75);
    \node[anchor=center] at (-2.5,0) {$k_C$};
    \filldraw[black] (-2.5,0.75) circle (0.75pt);
    \filldraw[black] (-2.5,0)+(90+360/5:0.75) circle (0.75pt);
    \filldraw[black] (-2.5,0)+(90+2*360/5:0.75) circle (0.75pt);
    \filldraw[black] (-2.5,0)+(90+3*360/5:0.75) circle (0.75pt);
    \filldraw[black] (-2.5,0)+(90+4*360/5:0.75) circle (0.75pt);
    \draw (-2.5,0)+(90+360/5:0.75) -- ($ (-2.5,0) + (306:0.75) $);
    \end{scope}
    \node[anchor=north] at (4,-1.5) {Type 2};

    \end{tikzpicture}
    \end{figure}

    \begin{figure}[H]
    \centering
    \begin{tikzpicture}

    \begin{scope}[shift={(1,-4)}]
    \draw (-2.5,0) ellipse (0.75 and 0.75);
    \node[anchor=center] at (-2.5,0) {$k_C$};
    \filldraw[black] (-2.5,0.75) circle (0.75pt);
    \filldraw[black] (-2.5,0)+(90+360/5:0.75) circle (0.75pt);
    \filldraw[black] (-2.5,0)+(90+2*360/5:0.75) circle (0.75pt);
    \filldraw[black] (-2.5,0)+(90+3*360/5:0.75) circle (0.75pt);
    \filldraw[black] (-2.5,0)+(90+4*360/5:0.75) circle (0.75pt);
    \filldraw[black] ($ (-2.5,0)+(90+4*360/5:0.75) $) +(22.5:0.75) circle (0.75pt);
    \draw ($ ($ (-2.5,0)+(90+4*360/5:0.75) $) +(22.5:0.75) $) -- ($ (-2.5,0)+(90+4*360/5:0.75) $);
    \end{scope}
    \node[anchor=north] at (-1.5,-5.5) {Type 3};

    \begin{scope}[shift={(5.5,-4)}]
    \draw (-2.5,0) ellipse (0.75 and 0.75);
    \node[anchor=center] at (-2.5,0) {$k_C$};
    \filldraw[black] (-2.5,0.75) circle (0.75pt);
    \filldraw[black] (-2.5,0)+(90+360/5:0.75) circle (0.75pt);
    \filldraw[black] (-2.5,0)+(90+2*360/5:0.75) circle (0.75pt);
    \filldraw[black] (-2.5,0)+(90+3*360/5:0.75) circle (0.75pt);
    \filldraw[black] (-2.5,0)+(90+4*360/5:0.75) circle (0.75pt);
    \draw (-0.5,0) ellipse (0.75 and 0.75);
    \node[anchor=center] at (-0.5,0) {$k_{C'}$};
    \begin{scope}[shift={(2,0)}]
    \filldraw[black] (-2.5,0.75) circle (0.75pt);
    \filldraw[black] (-2.5,0)+(90+360/5:0.75) circle (0.75pt);
    \filldraw[black] (-2.5,0)+(90+2*360/5:0.75) circle (0.75pt);
    \filldraw[black] (-2.5,0)+(90+3*360/5:0.75) circle (0.75pt);
    \filldraw[black] (-2.5,0)+(90+4*360/5:0.75) circle (0.75pt);
    \end{scope}
    \end{scope}
    \node[anchor=north] at (4,-5.5) {Type 4};

\end{tikzpicture}
\caption{The four types of graphs in $\mathcal{F}_C$.}\label{fig:4types}
\end{figure}
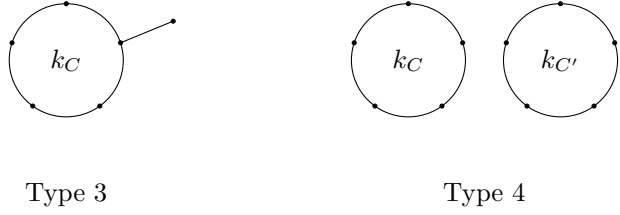

Let $\mathcal{F}=\bigcup_{r \geq 1, C \in S_r'} \mathcal{F}_C $. Then by property 3, every graph in $\mathcal{F}$ has at most two connected components, so $\mathcal{F}$ is quasi-finite. Also, by property 2, (c) holds. For each integer $r \geq 1$, let $S_r''$ be the set denoted by $S_r$ in the proof of part 1, that is, the set of residue classes $C$ modulo $r$ for which there exists an $\mathcal{F}$-saturated blueprint $(H,T)$ with $|H| \in C$ and $|T|=r$, and $S_r'''$ be the set of residue classes $C \in S_r''$ with $C \setminus \bigcup_{q<r, D \in S_q''} D \neq \emptyset$. By part 1, to prove (a) and (b) it suffices to show that $S_r'''=S_r'$ for every integer $r \geq 1$.  \newline

\emph{Claim 6:} For every integer $r \geq 1$ and residue class $C \in S_r'$, $C_{k_C} \cup kT_r$ is $\mathcal{F}$-saturated for every integer $k \geq 0$. In particular, $(C_{k_C}, T_r)$ is an $\mathcal{F}$-saturated blueprint, so $C \in S_r''$ by property 1. \newline

\emph{Proof:} We first show that $C_{k_C} \cup kT_r$ is $\mathcal{F}$-free. Note that for every integer $r' \geq 1$ and residue class $C' \in S_{r'}'$, every graph in $\mathcal{F}_{C'}$ contains $C_{k_{C'}}$. Hence, if $C_{k_C} \cup kT_r$ contains a graph in $\mathcal{F}_{C'}$, it contains  $C_{k_{C'}}$. But $C_{k_C}$ is the unique cycle in $C_{k_C} \cup kT_r$ and $k_{C'} \neq k_C$ if $C' \neq C$, so this can only happen if $r'=r$ and $C'=C$. Since $C_{k_C}$ is the unique cycle in $C_{k_C} \cup kT_r$, $C_{k_C} \cup kT_r$ does not contain a graph in $\mathcal{F}_C$ of the fourth type, since $C_{k_C}$ is a connected component of $C_{k_C} \cup kT_r$, $C_{k_C} \cup kT_r$ does not contain the graph in $\mathcal{F}_C$ of the third type, and since $C_{k_C}$ is an induced subgraph of $C_{k_C} \cup kT_r$, $C_{k_C} \cup kT_r$ does not contain a graph in $\mathcal{F}_C$ of the second type. Finally, by properties 3 and 4, $C_{k_C} \cup kT_r$ does not contain a graph in $\mathcal{F}_C$ of the first type. Hence $C_{k_C} \cup kT_r$ is $\mathcal{F}$-free.  \newline

It remains to show that every graph $\left(C_{k_C} \cup kT_r\right)+vw$ obtained by adding a new edge $vw$ to $C_{k_C} \cup kT_r$ contains a graph in $\mathcal{F}$. If $v,w \in C_{k_C}$, $\left(C_{k_C} \cup kT_r\right)+vw$ contains a graph in $\mathcal{F}_C$ of the second type, if $v \in C_{k_C}$ and $w \not \in C_{k_C}$, $\left(C_{k_C} \cup kT_r\right)+vw$ contains the graph in $\mathcal{F}_C$ of the third type, and if $v,w \not \in C_{k_C}$, $\left(C_{k_C} \cup kT_r \right)+vw$ contains a graph in $\mathcal{F}_C$ of the first type by property 5. \qedclaim\newline

\emph{Claim 7:} Let $(H,T)$ be an $\mathcal{F}$-saturated blueprint. Then either $H \cong C_{k_C} \cup kT_r$ and $T \cong T_r$ for some integers $r \geq 1$ and $k \geq 0$ and residue class $C \in S_r'$ or $S_r$ contains all of the residue classes modulo $r$ for some integer $1 \leq r \leq |T|$. \newline

\emph{Proof:} Let $m \geq 2$ be an integer such that $H \cup mT$ is $\mathcal{F}$-saturated. Let $T_1$ and $T_2$ be distinct connected components of $H \cup mT$ isomorphic to $T$ and pick vertices $u \in T_1$ and $w \in T_2$. Then $\left(H \cup mT \right)+uw$ contains a graph in $\mathcal{F}_C$ and hence $C_{k_C}$ for some integer $r \geq 1$ and residue class $C \in S_r'$. But there is no cycle in $\left(H \cup mT \right)+uw$ containing the edge $uw$, since $u$ and $w$ are in different connected components of $H \cup mT$, so in fact $C_{k_C} \subseteq H \cup mT$. Then $C_{k_C}$ must be an induced subgraph of $H \cup mT$, for otherwise $H \cup mT$ would contain a graph in $\mathcal{F}_C$ of the second type, and a connected component of $H \cup mT$, for otherwise $H \cup mT$ would contain the graph in $\mathcal{F}_C$ of the third type, and all other connected components of $H \cup mT$ must be $\mathcal{C}$-free, for otherwise $H \cup mT$ would contain a graph in $\mathcal{F}_C$ of the fourth type. \newline

Suppose first that $r \in \{1,2,3,4,5,7\}$. Since $C \in S_r'$, $S_r$ is non-empty, so contains all of the residue classes modulo $r$. Since there is no cycle in $\left(H \cup mT \right)+uw$ containing the edge $uw$ and all connected components of $H \cup mT$ except $C_{k_C}$ are $\mathcal{C}$-free, $C_{k_C}$ is the only copy of a graph in $\mathcal{C}$ in $\left(H \cup mT \right)+uw$. Hence $\left(H \cup mT \right)+uw$ does not contain a graph in $\mathcal{F}_C$ of the second, third or fourth type, so must contain a copy $F$ of a graph in $\mathcal{F}_C$ of the first type. Then $F$ must contain the edge $uw$, for otherwise $H \cup mT$ would contain $F$, and there is no cycle in $\left(H \cup mT \right)+uw$ containing the edge $uw$, so $(T_1 \cup T_2)+uw$ contains a graph in $\mathcal{F}_r$ by property 3. But every graph $T_r+e$ obtained by adding a new edge $e$ to $T_r$ contains a cycle and $(T_1 \cup T_2)+uw$ is a tree, so $(T_1 \cup T_2)+uw$ must contain a graph $2T_r+e$ obtained from $2T_r$ by adding an edge $e$ between the two copies of $T_r$. Hence $r \leq |T|$ by property 1. \newline

Now suppose $r \in \{6\} \cup \mathbb{Z}_{\geq 8}$. We show that all connected components $K \neq C_{k_C}$ of $H \cup mT$ are isomorphic to $T_r$ and hence that $H \cong C_{k_C} \cup kT_r$ and $T \cong T_r$ for some integer $k \geq 0$. We show that $K \in \mathcal{F}_r'$, which implies $K \cong T_r$, for otherwise $H \cup mT$ would contain a graph in $\mathcal{F}_C$ of the first type. Let $v \in K$. If $r \geq 6$ is even, we need to show that $v$ is adjacent to a vertex of degree at least $r/2$, and if $r \geq 9$ is odd, we need to show that $v$ is either adjacent to a vertex of degree at least $r-5$ or adjacent to a vertex of degree at least three which is adjacent to a vertex distinct from $v$ of degree at least three. \newline

Since $m \geq 2$, there is a connected component $T_3 \neq K$ of $H \cup mT$ isomorphic to $T$. Pick a vertex $\ell \in T_3$ of degree at most one. Then there is no cycle in $\left(H \cup mT\right)+v\ell$ containing the edge $v\ell$, since $v$ and $\ell$ are in different connected components of $H \cup mT$, and all connected components of $H \cup mT$ except $C_{k_C}$ are $\mathcal{C}$-free, so $C_{k_C}$ is the only copy of a graph in $\mathcal{C}$ in $\left(H \cup mT \right)+v\ell$. Hence $\left(H \cup mT \right)+v\ell$ does not contain a graph in $\mathcal{F}_{C'}$ for all $C' \neq C$, since all such graphs contain $C_{k_{C'}} \not \cong C_{k_C}$, and does not contain a graph in $\mathcal{F}_C$ of the second, third or fourth type. So $\left(H \cup mT \right)+v\ell$ must contain a copy $F$ of a graph in $\mathcal{F}_C$ of the first type. Then $F$ must contain the edge $v\ell$, for otherwise $H \cup mT$ would contain $F$, and there is no cycle in $\left(H \cup mT \right)+v\ell$ containing the edge $v\ell$, so $(K \cup T_3)+v\ell$ contains a copy of a graph in $\mathcal{F}_r$ that contains $v$ by property 3. \newline

Hence, if $r \geq 6$ is even, $v$ is adjacent to a vertex of degree at least $r/2$ in $(K \cup T_3)+v\ell$, and if $r \geq 9$ is odd, $v$ is either adjacent to a vertex of degree at least $r-5$ or adjacent to a vertex of degree at least three which is adjacent to a vertex distinct from $v$ of degree at least three in $(K \cup T_3)+v\ell$. But $\ell$ has degree at most two in $(K \cup T_3)+v\ell$, $r/2 \geq 3$ if $r \geq 6$ is even and $r-5 \geq 4$ if $r\geq 9$ is odd, so, if $r \geq 6$ is even, $v$ is adjacent to a vertex of degree at least $r/2$ in $K$, and if $r \geq 9$ is odd, $v$ is either adjacent to a vertex of degree at least $r-5$ or adjacent to a vertex of degree at least three which is adjacent to a vertex distinct from $v$ of degree at least three in $K$. \qedclaim\newline

Using Claims 6 and 7, property 1 and induction on $r$, it is easy to show that $S_r'''=S_r'$ for every integer $r \geq 1$. \newline

\end{proof}

\section{Singleton families}\label{sin}

In this section we discuss the error in \cite{TrusTuza} and prove Theorem~\ref{sinfam}. To prove Theorem~\ref{sinfam}, we will need the following key lemma.

\begin{lemma}\label{lem:sinfam}
Let $F$ be a graph and $T$ be a tree such that every graph $T+e$ obtained by adding a new edge $e$ to $T$ and every graph $2T+e$ obtained from $2T$ by adding an edge $e$ between the two copies of $T$ contains a connected component of $F$ that is not contained in $T$. Then for every residue class $C$ modulo $|T|$ there is an $F$-saturated blueprint $(H,T)$ with $|H| \in C$.     
\end{lemma}

We first prove Theorem \ref{sinfam} using Lemma \ref{lem:sinfam}.

\begin{repth3}

For every graph $F$, either there is an integer $r \geq 1$ and an integer $k_C$ for each residue class $C$ modulo $r$ such that 

$$ sat(n,F)= \frac{(r-1)n+k_C}{r} $$
        
for every residue class $C$ modulo $r$ and large enough $n \in C$ or

$$ sat(n,F) \geq n+o(n) . $$ 

The former occurs if and only if there is a tree $T$ such that every graph $T+e$ obtained by adding a new edge $e$ to $T$ and every graph $2T+e$ obtained from $2T$ by adding an edge $e$ between the two copies of $T$ contains a connected component of $F$ that is not contained in $T$ and $r$ is the minimum number of vertices in such a tree.
\end{repth3}

\begin{proof}

Let $F$ be a graph. Apply part 1 of Theorem \ref{quasfinfam} to the family $\mathcal{F}=\{F\}$ and recall from the proof of Theorem \ref{quasfinfam} that for every integer $r \geq 1$, $S_r$ is the set of residue classes $C$ modulo $r$ for which there exists an $F$-saturated blueprint $(H,T)$ with $|H| \in C$ and $|T|=r$. \newline 

\emph{Claim 1:} If $(H,T)$ is an $F$-saturated blueprint then every graph $T+e$ obtained by adding a new edge $e$ to $T$ and every graph $2T+e$ obtained from $2T$ by adding an edge $e$ between the two copies of $T$ contains a connected component of $F$ that is not contained in $T$.\newline

\emph{Proof:} 
Let $k_0$ be the number of connected components of $F$ which are contained in $T$. As $(H,T)$ is an $F$-saturated blueprint, there exists an integer $k\ge\max\{k_0,2\}$ such that $G:=H\cup kT$ is $F$-saturated. Suppose for the sake of contradiction that the claim does not hold. Then, there exists a missing edge $e$ in $kT\subseteq G$ such that adding $e$ to $G$ does not create a connected component of $F$ that is not contained in $T$. As $G$ is $F$-saturated, $G+e$ contains a copy of $F$. All connected components of this copy which are not contained in $T$ must lie in $H$, as adding $e$ did not create any such connected components. Finally, $kT$ contains the $k_0$ connected components of $F$ which are contained in $T$, so $G$ contains $F$, a contradiction.\qedclaim\newline

Combining Claim 1 and Lemma~\ref{lem:sinfam}, we obtain that for all integers $r \geq 1$, $S_r$ contains either none or all of the residue classes modulo $r$, with the latter occurring if and only if there is a tree $T$ on $r$ vertices such that every graph $T+e$ obtained by adding a new edge $e$ to $T$ and every graph $2T+e$ obtained from $2T$ by adding an edge $e$ between the two copies of $T$ contains a connected component of $F$ that is not contained in $T$. This proves the theorem.
\end{proof}

We now discuss the error in~\cite{TrusTuza}. So far, our approach to proving Theorem~\ref{sinfam} is essentially the same as that in~\cite{TrusTuza} (though only singleton families are considered there, so there is no result like part 1 of Theorem \ref{quasfinfam}) and it is in the proof of the analogue of Lemma~\ref{lem:sinfam} that the mistake arises. More precisely, the error is at the top of the fourth page (labelled as page 312), where it is incorrectly claimed that a certain graph $H_{a,b}$ is $F$-saturated (see Remark~\ref{rmk:counterexample} for an explicit example where $H_{a,b}$ is not $F$-saturated). \newline

The approach taken there to prove Lemma~\ref{lem:sinfam} is to take $H$ to be the disjoint union of cliques of carefully chosen sizes. Before giving a correct proof of Lemma~\ref{lem:sinfam}, we provide an example where we cannot take $H$ to be a disjoint union of cliques. This shows that a novel construction is required to prove Lemma~\ref{lem:sinfam}, rather than just a different choice of the sizes of the cliques.

\begin{exmp}\label{counterexample}
Let $T$ be the star on eight vertices. Let $F_\triangle$, $F_{CC}$, $F_{CL}$ and $F_{LL}$ denote the following graphs, respectively: a triangle, two disjoint copies of $T$ plus an edge between their centers, two disjoint copies of $T$ plus an edge between the center of one copy and a leaf of the other copy, and two disjoint copies of $T$ plus an edge between two leaves of different copies. Let $F:=F_\triangle\cup F_{CC}\cup F_{CL}\cup F_{LL}$. Then $F$ and $T$ satisfy the assumptions of Lemma~\ref{lem:sinfam}. Furthermore, if $H$ is a disjoint union of cliques and $(H,T)$ is an $F$-saturated blueprint, then $|H|\not\equiv-1 \mod 8$. 
\end{exmp}

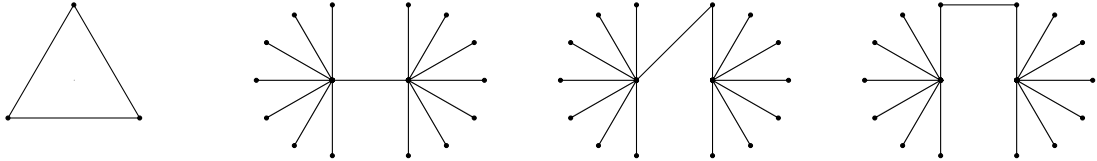
\begin{figure}[H]
\centering
\begin{tikzpicture}

         \filldraw (-3.4,0) circle (0pt) +(-30:1) circle (0.75pt) +(-150:1) circle (0.75pt) +(-270:1) circle (0.75pt);

         \draw (-3.4,0) +(-30:1)-- +(-150:1)-- +(-270:1)-- +(-30:1);

         \draw (0,0) -- (1,0);

         \filldraw (0,0) circle (0.75pt) -- +(90:1) circle (0.75pt)
         +(0,0) circle (0.75pt) -- +(120:1) circle (0.75pt)
         +(0,0) circle (0.75pt) -- +(150:1) circle (0.75pt)
         +(0,0) circle (0.75pt) -- +(180:1) circle (0.75pt)
         +(0,0) circle (0.75pt) -- +(210:1) circle (0.75pt)
         +(0,0) circle (0.75pt) -- +(240:1) circle (0.75pt)
         +(0,0) circle (0.75pt) -- +(270:1) circle (0.75pt);

         \filldraw (1,0) circle (0.75pt) -- +(-90:1) circle (0.75pt)
         +(0,0) circle (0.75pt) -- +(-60:1) circle (0.75pt)
         +(0,0) circle (0.75pt) -- +(-30:1) circle (0.75pt)
         +(0,0) circle (0.75pt) -- +(0:1) circle (0.75pt)
         +(0,0) circle (0.75pt) -- +(30:1) circle (0.75pt)
         +(0,0) circle (0.75pt) -- +(60:1) circle (0.75pt)
         +(0,0) circle (0.75pt) -- +(90:1) circle (0.75pt);

         \draw (4,0) -- (5,1);

         \filldraw (4,0) circle (0.75pt) -- +(90:1) circle (0.75pt)
         +(0,0) circle (0.75pt) -- +(120:1) circle (0.75pt)
         +(0,0) circle (0.75pt) -- +(150:1) circle (0.75pt)
         +(0,0) circle (0.75pt) -- +(180:1) circle (0.75pt)
         +(0,0) circle (0.75pt) -- +(210:1) circle (0.75pt)
         +(0,0) circle (0.75pt) -- +(240:1) circle (0.75pt)
         +(0,0) circle (0.75pt) -- +(270:1) circle (0.75pt);

         \filldraw (5,0) circle (0.75pt) -- +(-90:1) circle (0.75pt)
         +(0,0) circle (0.75pt) -- +(-60:1) circle (0.75pt)
         +(0,0) circle (0.75pt) -- +(-30:1) circle (0.75pt)
         +(0,0) circle (0.75pt) -- +(0:1) circle (0.75pt)
         +(0,0) circle (0.75pt) -- +(30:1) circle (0.75pt)
         +(0,0) circle (0.75pt) -- +(60:1) circle (0.75pt)
         +(0,0) circle (0.75pt) -- +(90:1) circle (0.75pt);

         \draw (8,1) -- (9,1);

         \filldraw (8,0) circle (0.75pt) -- +(90:1) circle (0.75pt)
         +(0,0) circle (0.75pt) -- +(120:1) circle (0.75pt)
         +(0,0) circle (0.75pt) -- +(150:1) circle (0.75pt)
         +(0,0) circle (0.75pt) -- +(180:1) circle (0.75pt)
         +(0,0) circle (0.75pt) -- +(210:1) circle (0.75pt)
         +(0,0) circle (0.75pt) -- +(240:1) circle (0.75pt)
         +(0,0) circle (0.75pt) -- +(270:1) circle (0.75pt);

         \filldraw (9,0) circle (0.75pt) -- +(-90:1) circle (0.75pt)
         +(0,0) circle (0.75pt) -- +(-60:1) circle (0.75pt)
         +(0,0) circle (0.75pt) -- +(-30:1) circle (0.75pt)
         +(0,0) circle (0.75pt) -- +(0:1) circle (0.75pt)
         +(0,0) circle (0.75pt) -- +(30:1) circle (0.75pt)
         +(0,0) circle (0.75pt) -- +(60:1) circle (0.75pt)
         +(0,0) circle (0.75pt) -- +(90:1) circle (0.75pt);

\end{tikzpicture}
\caption{The graph $F$.}\label{fig:counterexample}
\end{figure}

\begin{proof}

It is easy to see that adding a missing edge to $T$ creates a copy of $F_{\triangle}$. Similarly, adding an edge between two disjoint copies of $T$ creates a copy of $F_{CC}$, $F_{CL}$ or $F_{LL}$. Hence, $F$ and $T$ satisfy the assumptions of Lemma~\ref{lem:sinfam}.\newline

Suppose $H$ is a disjoint union of cliques and $(H,T)$ is an $F$-saturated blueprint. Thus, there exists an integer $k\ge1$ such that $H\cup kT$ is $F$-saturated. Note that adding an edge between two leaves of a copy of $T$ creates a copy of $F_\triangle$ but no other connected component of $F$. Since adding a missing edge to $H\cup kT$ creates a copy of $F$, this implies that $H$ contains $F\setminus F_{\triangle}=F_{CC}\cup F_{CL}\cup F_{LL}$. \newline

Next, note that $H$ does not have connected components of size one or two, as adding an edge between such connected components and a copy of $T$ would not create a connected component of $F$ (and so it would not create a copy of $F$). Hence, every connected component of $H$ contains $F_\triangle$. Since $F\setminus F_\triangle$ has exactly three connected components and $H$ contains $F\setminus F_\triangle$, it follows that $H$ has at most three connected components, as otherwise $H$ (and hence $H\cup kT$) would contain $F$. \newline 

Finally, we bound the size of $H$ from above and below. First, since $H$ contains $F\setminus F_\triangle$, we have $|H|\ge|F\setminus F_\triangle| = 48$. On the other hand, fixing a copy of $F\setminus F_\triangle$ in $H$, if some connected component of $H$ contains three vertices outside the copy, then $H$ (and hence $H\cup kT$) contains $F$, a contradiction. Thus, each connected component of $H$ contains at most two vertices outside the copy of $F\setminus F_\triangle$. As $H$ has at most three connected components, this implies $|H|\le |F\setminus F_\triangle|+6 = 54$. Hence $48\le|H|\le54$. In particular, $|H|\not\equiv-1 \mod 8$, as required.
\end{proof}

\begin{rem}\label{rmk:counterexample}
For $F$ and $T$ as in Example~\ref{counterexample}, the graph $H_{a,b}$ claimed to be $F$-saturated in~\cite{TrusTuza} is $aK_{15} \cup bK_{14} \cup 4(K_{15}\cup T) \cup K_{47}$, where $a,b \geq 0$ are integers. Observe that $H_{a,b}$ does not contain $F_{CC}\cup F_{CL}\cup F_{LL}$, so it is $F$-free, but adding a new edge $e$ between two leaves of a copy of $T$ does not create a copy of $F_{CC}$, $F_{CL}$ or $F_{LL}$. Hence $H_{a,b}+e$ is still $F$-free, so $H_{a,b}$ is not $F$-saturated. 
\end{rem}

\begin{rem}
Note that to prove Theorem \ref{sinfam} a weaker version of Lemma~\ref{lem:sinfam} would suffice where we additionally assume that $T$ is the smallest tree satisfying the assumptions. However, Example~\ref{counterexample} still applies: it is easy to check that $T$ is the unique smallest tree such that the assumptions of Lemma~\ref{lem:sinfam} hold.
\end{rem}

We now prove Lemma \ref{lem:sinfam}. 

\begin{proof}[Proof of Lemma~\ref{lem:sinfam}]
Let $F^*$ be the union of all connected components of $F$ which are not contained in $T$. It is straightforward to check that a blueprint $(H,T)$ is $F$-saturated if and only if it is $F^*$-saturated. Therefore, we may assume that no connected component of $F$ is contained in $T$. Also, $F$ must have a connected component that is a tree, as adding an edge between two disjoint copies of $T$ creates a connected component of $F$.\newline

\emph{Claim 1:} For every integer $k \geq 0$, $K_{|F|-1} \cup kT$ is $F$-saturated. In particular, $(K_{|F|-1},T)$ is an $F$-saturated blueprint.\newline

\emph{Proof:}
Note that $K_{|F|-1}\cup kT$ is $F$-free, since $K_{|F|-1}$ is $F$-free and no connected component of $F$ is contained in $T$. It remains to show that adding a missing edge $xy$ to $K_{|F|-1}\cup kT$ creates a copy of $F$. If $x$ and $y$ belong to one or two copies of $T$, say $T_1$ and $T_2$, then $\left(T_1\cup T_2\right)+xy$ contains a connected component $F_0$ of $F$ by assumption. As $K_{|F|-1}$ contains $F\setminus F_0$, it follows that $K_{|F|-1} \cup \left[ \left(T_1\cup T_2\right)+xy \right]$ (and hence $(K_{|F|-1}\cup kT)+xy$) contains $F$. Suppose instead that, say, $x$ belongs to $K_{|F|-1}$. Recall that $F$ has a connected component that is a tree, meaning that $F$ has a vertex of degree one. Thus, $(K_{|F|-1}\cup\{y\})+xy$ (and hence $(K_{|F|-1}\cup kT)+xy$) contains $F$, as required.
\qedclaim\newline

If $|T|=1$ then Claim 1 proves Lemma \ref{lem:sinfam}, so suppose that $|T|\ge2$ and let $C$ be a residue class modulo $|T|$. First, we consider an easier case.\newline

\emph{Claim 2:} If a connected component of $F$ is a star then there exists an $F$-saturated blueprint $(H,T)$ with $|H|\in C$.\newline

\emph{Proof:}
Suppose that a connected component $F_0$ of $F$ is a star on, say, $\ell+1$ vertices. Let $m \geq 0$ be an integer such that $|F|-1+m\in C$ and let $H^*:= K_{|F|-1}\cup I_m\cup(\ell m-m+1)T$, where $I_m$ is an independent set of size $m$. Note that $H^*$ is $F$-free, since $K_{|F|-1}$ is $F$-free and no connected component of $F$ is contained in $T$. Hence we can add edges to $H^*$ to obtain an $F$-saturated graph $H$. \newline

Note that all new edges are incident to $I_m$, as otherwise $H$ would contain $\left[ K_{|F|-1} \cup (\ell m -m+1)T \right] +e$ for some missing edge $e$, and hence $F$ by Claim 1. Note also that each vertex in $I_m$ has at most $\ell-1$ neighbours in $(\ell m-m+1)T$, as otherwise $H$ would contain a copy of $F_0$ disjoint from $K_{|F|-1}$, and thus a copy of $F$. Hence, there are at most $\ell m-m$ vertices in $(\ell m-m+1)T$ which are adjacent to some vertex in $I_m$. In particular, there is a copy of $T$ in $(\ell m-m+1)T$ which is a connected component in $H$.

\begin{figure}[H]
\centering
\begin{tikzpicture}

\fill[black!15] (3,0.3) rectangle (3,-0.3); 
\fill[black!15] (3,0.2)--(3,-0.2)--(6,1.8)--(6,2.2)--(3,0.2); 
\fill[black!15] (3,0.2)--(3,-0.2)--(6,-0.7)--(6,-0.3)--(3,0.2); 
%\fill[black!15] (3,0.2)--(3,-0.2)--(6,-2.2)--(6,-1.8)--(3,0.2); 
\fill[black!15] (0,0.5) rectangle (3,-0.5); 
\filldraw[fill=black!0] (0,0) circle (1cm) node[black]{$K_{|F|-1}$};
\filldraw[fill=black!15] (3,0) circle (1cm) node[black]{$I_m$};  
\filldraw[fill=black!0] (6,-2) circle (0.6cm) node[black]{$T$};
\filldraw[fill=black!0] (6,-0.5) circle (0.6cm) node[black]{$T$};
\filldraw[fill=black!0] (6,0.75) circle (0cm) node[black]{$\vdots$};
\filldraw[fill=black!0] (6,2) circle (0.6cm) node[black]{$T$};
\draw [decorate,decoration={brace,amplitude=10pt}] (6.9,2.6)--(6.9,-2.6) node[midway,right=15pt] {$\ell m-m+1$};
\end{tikzpicture}
\caption{The graph $H$. The shading denotes the presence of new edges.}\label{fig:starcase}
\end{figure}
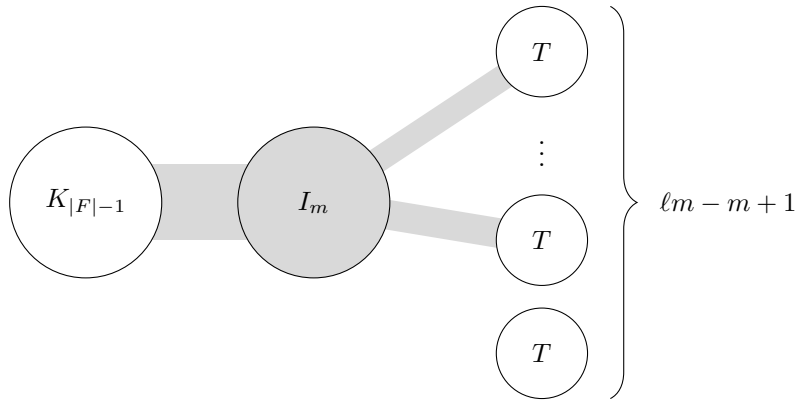

We have $|H|\in C$ and so it remains to show that $(H,T)$ is an $F$-saturated blueprint. We in fact show that $H\cup kT$ is $F$-saturated for every integer $k\ge0$. First, observe that $H$ being $F$-free implies $H\cup kT$ is $F$-free, since no connected component of $F$ is contained in $T$. Let $e$ be a missing edge in $H\cup kT$. If $e$ is not incident to $I_m$ then $(H\cup kT)+e$ contains $\left [ K_{|F|-1}\cup (\ell m-m+1+k)T \right]+e$ and hence $F$ by Claim 1. If $e$ is incident to $I_m$ then $(H\cup kT)+e$ contains a copy of $H+e'$ for some missing edge $e'$ (here we used the fact that there is a copy of $T$ in $(\ell m-m+1)T$ which is a connected component in $H$), and hence $F$, since $H$ is $F$-saturated. This concludes the proof.
\qedclaim\newline

Thus, we may assume that no connected component of $F$ is a star. Given a graph $G$, we say $S\subseteq V(G)$ is a {\it cover} of $G$ if every edge of $G$ is incident to $S$. We let $\sigma(G)$ be the size of the smallest cover of $G$. Pick an integer $m \geq |F|$ such that $\sigma(F)-1+m \in C$. Let $H^*$ be the graph consisting of two disjoint sets of vertices $A$ and $B$, of size $\sigma(F)-1$ and $m$ respectively, and all edges incident to $A$. Since $\sigma(H^*)=|A|=\sigma(F)-1 < \sigma(F)$, $H^*$ is $F$-free. Hence we can add edges to $H^*$ to obtain an $F$-saturated graph $H$. \newline

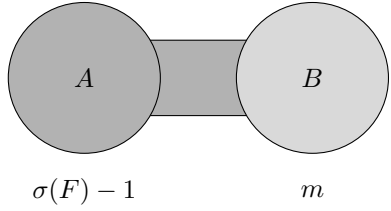
\begin{figure}[H]
\centering
\begin{tikzpicture}
\scalebox{1}{

\filldraw[fill=black!30] (0.5,0.5)--(2.5,0.5)--(2.5,-0.5)--(0.5,-0.5);
\filldraw[fill=black!30] (0,0) circle (1);
\node[anchor=center] at (0,0) {\scalebox{1}{$A$}};
\filldraw[fill=black!15] (3,0) circle (1);
\node[anchor=center] at (3,0) {\scalebox{1}{$B$}};
\node[anchor=center] at (0,-1.5) {\scalebox{1}{$\sigma(F)-1$}};
\node[anchor=center] at (3,-1.5) {\scalebox{1}{$m$}};

}
\end{tikzpicture}
\caption{The graph $H$. The dark and light shades denote the presence of all and some edges, respectively.}\label{fig:starcase}
\end{figure}

Note that $|H| \in C$, so it remains to prove that $(H,T)$ is an $F$-saturated blueprint. We in fact prove that $H\cup kT$ is $F$-saturated for every integer $k\ge0$. First, note that $H\cup kT$ is $F$-free, since $H$ is $F$-free and no connected component of $F$ is contained in $T$. It remains to show that $(H\cup kT)+xy$ contains $F$ for every missing edge $xy$. There are four cases to consider. The first two are straightforward, while the latter two require the power of our construction.\newline

{\it Case 1:} $x,y\in H$. \newline

As $H$ is $F$-saturated, $H+xy$ (and hence $(H\cup kT)+xy$)  contains $F$. \newline

{\it Case 2:} $x,y\not\in H$. \newline

Then $x$ and $y$ must lie in one or two copies of $T$, say $T_1$ and $T_2$. By assumption, $(T_1\cup T_2)+xy$ contains a connected component $F_0$ of $F$. Since $\sigma(F\setminus F_0)=\sigma(F)-\sigma(F_0)\le\sigma(F)-1 = |A|$ and $|B|=m \geq |F|$, $H^*$ contains $F\setminus F_0$, so $H^* \cup \left[ (T_1\cup T_2)+xy \right]$ (and hence $(H\cup kT)+xy$) contains $F$.\newline

{\it Case 3:} $x\in A$ and $y\not\in H$. \newline

We have that $y$ lies in some copy $T_y$ of $T$. Let $T_1$ and $T_2$ be two trees isomorphic to $T$. We have graph isomorphisms $f_1:T_y\to T_1$ and $f_2:T_y\to T_2$. Let $y_1:=f_1(y)$ and  $y_2:=f_2(y)$ be the copies of $y$ in $T_1$ and $T_2$, respectively. By assumption, $(T_1\cup T_2)+y_1y_2$ contains a copy $F_0$ of a connected component of $F$. Note that $F_0$ must contain the edge $y_1y_2$ (and hence the vertices $y_1$ and $y_2$), as otherwise $F_0$ would be contained in $T$. \newline

Let $C_1 \subseteq T_1$ and $C_2 \subseteq T_2$ be the connected components of $F_0-y_1y_2$ containing $y_1$ and $y_2$, respectively. Let $S$ be a cover of $F_0$ of minimum size, i.e. with $|S|=\sigma(F_0)$. By the definition of a cover, we have $\{y_1,y_2\}\cap S\not=\emptyset$. Without loss of generality, $y_1\in S$. Now we explain how to embed $F$ in $(H^*\cup T_y)+xy$ (and hence in $(H\cup kT)+xy$). Note that $F \cong \left[ (F\setminus F_0)\cup C_1\cup C_2 \right] +y_1y_2$. We will embed $F \setminus F_0$, $C_1$, $C_2$ and $y_1y_2$ separately. \newline

First, suppose that $|C_2|\ge2$, which implies $|S \cap C_2|\ge1$. Embed $F\setminus F_0$ and $C_1$ in $H^*$. This is possible since $S \cap C_1$ is a cover of $C_1$,  

$$\sigma(F\setminus F_0)+|S \cap C_1|=\sigma(F)-|S \cap C_2|\le\sigma(F)-1=|A| $$

and $|B|=m \geq |F|$. As $y_1\in S \cap C_1$ and $x\in A$, we can choose to send $y_1$ to $x$. Next, embed $C_2$ in $T_y$ via $f_2^{-1}$. In particular, $y_2$ is sent to $y$. Finally, embed the edge $y_1y_2$ as $xy$.\newline

Next, suppose that $|C_2|=1$, which implies $C_2=\{y_2\}$. Since $F_0$ is not a star, it follows that $\sigma(F_0)\ge2$. First, embed $C_1$ in $T_y$ via $f_1^{-1}$. In particular, $y_1$ is sent to $y$. Next, embed $C_2=\{y_2\}$ as $x$. Then, embed $F \setminus F_0$ in $H^* - x$. This is possible since

$$\sigma(F\setminus F_0)=\sigma(F)-\sigma(F_0) \le \sigma(F)-2=|A \setminus \{x\}| $$

and $|B|=m \geq |F|$. Finally, embed the edge $y_1y_2$ as $yx$. \newline

\begin{figure}[H]
\centering
\begin{tikzpicture}
\fill[black!30] (-1.5,0)--(1.5,0)--(1.5,2)--(-1.5,2)--(-1.5,0);
\filldraw[fill=black!30] (0,0) ellipse (2cm and 0.75cm);
\filldraw[fill=black!00] (0,2) ellipse (2cm and 0.75cm);
\draw (0,0) circle (0) node{$A$};
\draw (0,2) circle (0) node{$B$};

\filldraw[fill=blue!30] (-0.75,1) ellipse (0.6cm and 1cm) node[black]{$F\setminus F_0$};
\filldraw[fill=blue!30] (0.75,1) ellipse (0.6cm and 1cm) node[black]{$C_1$};

\filldraw[fill=blue!30] (3.5,1) circle (0.75cm) node[black,anchor=south]{$C_2$};
\draw (3.5,1) circle (1.3cm);
\draw (3.5,2) circle (0) node{$T_y$};

\draw[very thick,blue!60] (0.75,0.5)--(3.5,0.75);
\fill[red] (0.75,0.5) circle (2.5pt) node[black,anchor=north]{$y_1$};
\fill[yellow] (3.4,0.65) rectangle (3.6,0.85); \node[anchor=west] at (3.5,0.7) {$y_2$};
%\fill[yellow] (3.5,0.75) circle (2.5pt) node[black,anchor=north]{$y_2$};

%%%%%%

\fill[black!30] (7.5,0)--(10.5,0)--(10.5,2)--(7.5,2)--(7.5,0);
\filldraw[fill=black!30] (9,0) ellipse (2cm and 0.75cm);
\filldraw[fill=black!00] (9,2) ellipse (2cm and 0.75cm);
\draw (9,0) circle (0) node{$A$};
\draw (9,2) circle (0) node{$B$};

\filldraw[fill=blue!30] (8.25,1) ellipse (0.6cm and 1cm) node[black]{$F\setminus F_0$};

\filldraw[fill=blue!30] (12.5,1) circle (0.75cm) node[black,anchor=south]{$C_1$};
\draw (12.5,1) circle (1.3cm);
\draw (12.5,2) circle (0) node{$T_y$};

\draw[very thick,blue!60] (9.75,0.5)--(12.5,0.75);
\fill[red] (9.75,0.5) circle (2.5pt) node[black,anchor=north]{$y_2$};
%\fill[yellow] (12.5,0.75) circle (2.5pt) node[black,anchor=north]{$y_1$};

\fill[yellow] (12.4,0.65) rectangle (12.6,0.85); 
\node[anchor=west] at (12.5,0.75) {$y_1$};
\end{tikzpicture}
\caption{The embeddings in Case 3. The subcases $|C_2|\ge2$ and $|C_2|=1$ are on the left and right, respectively. The red circle and the yellow square represent $x$ and $y$, respectively.}\label{fig:cover3.2}
\end{figure}
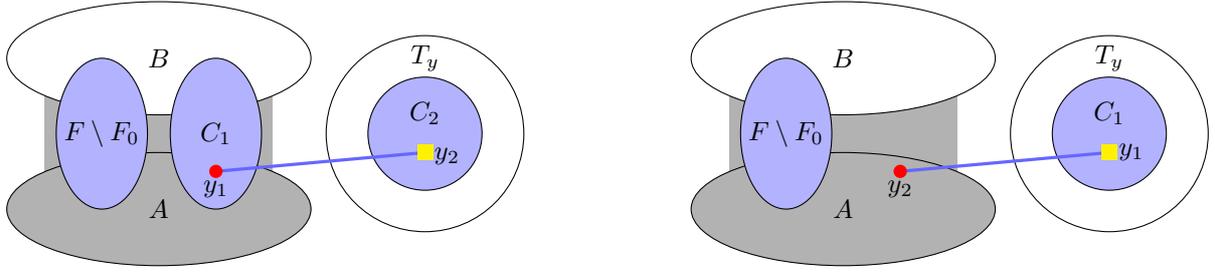

{\it Case 4:} $x\in B$ and $y\not\in H$. \newline

We have that $y$ lies in some copy $T_y$ of $T$. Let $T_1$ and $T_2$ be two trees isomorphic to $T$. We have a graph isomorphism $f:T_y\to T_1$. Let $y_1:=f(y)$ be the copy of $y$ in $T_1$, $\ell \in T_2$ be a leaf (note $|T| \geq 2$) and $\ell'\in T_2$ be the unique neighbour of $\ell$. By assumption, $(T_1\cup T_2)+y_1\ell$ contains a copy $F_0$ of a connected component of $F$. Note that $F_0$ must contain the edge $y_1l$ (and hence the vertices $y_1$ and $\ell$), as otherwise $F_0$ would be contained in $T$. Let $C_1 \subseteq T_1$ and $C_2 \subseteq T_2$ be the connected components of $F_0-y_1l$ containing $y_1$ and $\ell$, respectively. \newline

Let $S$ be a cover of $F_0$ of minimum size, i.e. with $|S|=\sigma(F_0)$. Then $S':=\left[ (S \setminus \{\ell\} ) \cup \{\ell'\} \right] \cap C_2$ is a cover of $C_2$. If $\ell' \in C_2 \setminus S$, the edge $\ell\ell'$ is in $F_0$, so $\ell \in S$. Hence $|S'|\le|S \cap C_2| \leq |S|$. We have equality if and only if $S \subseteq C_2$ and either $\ell' \in C_2 \setminus S$ or $\ell \not \in S$. Since the edge $y_1l$ is in $F_0$, this happens if and only if $C_1=\{y_1\}$, the edge $\ell\ell'$ is in $F_0$, $y_1 \not \in S$, $\ell \in S$ and $\ell' \not \in S$. Now we explain how to embed $F$ in $(H^* \cup T_y)+xy$ (and hence in $(H\cup kT)+xy$).  \newline

Suppose first that $|S'| \leq |S|-1$. Note that $F \cong \left[ (F\setminus F_0)\cup C_1\cup C_2 \right] +y_1\ell$. We will embed $F \setminus F_0$, $C_1$, $C_2$ and $y_1l$ separately.  Embed $F\setminus F_0$ and $C_2$ in $H^*$. This is possible since $S'$ is a cover of $C_2$,

$$\sigma(F\setminus F_0)+|S'| \le \sigma(F)-1=|A| $$ 

and $|B|=m \geq |F|$. As $\ell \in C_2 \setminus S'$ and $x\in B$, we can choose to send $\ell$ to $x$. Next, embed $C_1$ in $T_y$ via $f^{-1}$. In particular, $y_1$ is sent to $y$. Finally, embed the edge $y_1\ell$ as $yx$.\newline

Now suppose $|S'|=|S|$ or equivalently that $C_1=\{y_1\}$, the edge $\ell\ell'$ is in $F_0$, $y_1 \not \in S$, $\ell \in S$ and $\ell' \not \in S$. Note that $F \cong \left[ (F\setminus F_0) \cup (F_0 \setminus \{y_1,\ell\}) \cup \{y_1\} \cup \{\ell\} \right] + y_1l + \ell\ell'$. We will embed $F \setminus F_0$, $F_0 \setminus \{y_1,\ell\}$, $y_1$, $\ell$, $y_1\ell$ and $\ell\ell'$ separately. Embed $F\setminus F_0$ and $F_0 \setminus \{y_1,\ell\}$ in $H^*$. This is possible since $S \setminus \{\ell\}$ is a cover of $F_0 \setminus \{y_1,\ell\}$,

$$\sigma(F \setminus F_0) +|S \setminus \{\ell\}| = \sigma(F)-1 = |A| $$

and $|B|=m \geq |F|$. As $\ell' \in F_0 \setminus \left(S \cup \{y_1\} \right)$ and $x\in B$, we can choose to send $\ell'$ to $x$. Since $|T| \geq 2$, $y$ has a neighbour $y' \in T_y$. Embed $y_1$, $\ell$, $y_1\ell$ and $\ell\ell'$ as $y'$, $y$, $y'y$ and $yx$, respectively.

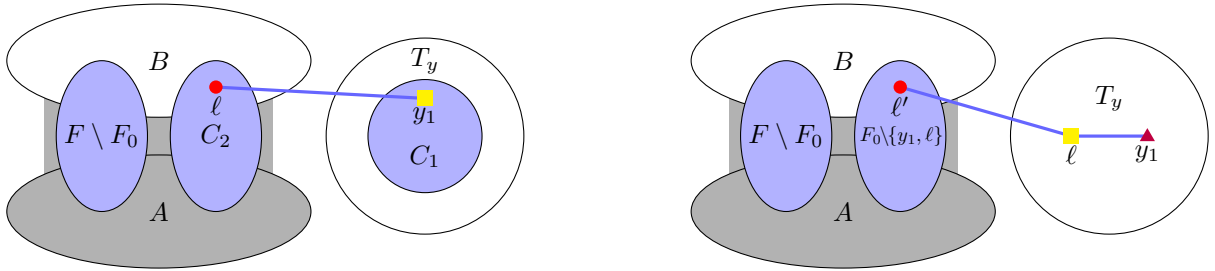
\begin{figure}[H]
\centering
\begin{tikzpicture}
\fill[black!30] (-1.5,0)--(1.5,0)--(1.5,2)--(-1.5,2)--(-1.5,0);
\filldraw[fill=black!30] (0,0) ellipse (2cm and 0.75cm);
\filldraw[fill=black!00] (0,2) ellipse (2cm and 0.75cm);
\draw (0,0) circle (0) node{$A$};
\draw (0,2) circle (0) node{$B$};

\filldraw[fill=blue!30] (-0.75,1) ellipse (0.6cm and 1cm) node[black]{$F\setminus F_0$};
\filldraw[fill=blue!30] (0.75,1) ellipse (0.6cm and 1cm) node[black]{$C_2$};

\filldraw[fill=blue!30] (3.5,1) circle (0.75cm) node[black,anchor=north]{$C_1$};
\draw (3.5,1) circle (1.3cm);
\draw (3.5,2) circle (0) node{$T_y$};

\draw[very thick,blue!60] (0.75,1.65)--(3.5,1.5);
\fill[red] (0.75,1.65) circle (2.5pt) node[black,anchor=north]{$\ell$};
%\fill[yellow] (3.5,1.5) circle (2.5pt) node[black,anchor=north]{$y_1$};

\fill[yellow] (3.4,1.4) rectangle (3.6,1.6); 
\node[anchor=north] at (3.5,1.5) {$y_1$};

%%%%%

\fill[black!30] (7.5,0)--(10.5,0)--(10.5,2)--(7.5,2)--(7.5,0);
\filldraw[fill=black!30] (9,0) ellipse (2cm and 0.75cm);
\filldraw[fill=black!00] (9,2) ellipse (2cm and 0.75cm);
\draw (9,0) circle (0) node{$A$};
\draw (9,2) circle (0) node{$B$};

\filldraw[fill=blue!30] (8.25,1) ellipse (0.6cm and 1cm) node[black]{$F\setminus F_0$};
\filldraw[fill=blue!30] (9.75,1) ellipse (0.6cm and 1cm);

%\node[anchor=center] at (9.74,0.75) {$\underset{\mbox{ \normalsize $\{y_1,l\}$}}{F_0 \setminus}$};

\node[anchor=center] at (9.77,1) {\scalebox{0.75}{$F_0 \!\!\setminus\!\!\{y_1,\ell\}$}};

%\filldraw[fill=blue!30] (12.5,1) circle (0.5cm) node[black,anchor=north]{$C_{y_1}$};
\draw (12.5,1) circle (1.3cm);
\draw (12.5,1.5) circle (0) node{$T_y$};

\draw[very thick,blue!60] (9.75,1.65)--(12,1)--(13,1);
\fill[red] (9.75,1.65) circle (2.5pt) node[black,anchor=north]{$\ell'$};
%\fill[yellow] (12,1) circle (2.5pt) node[black,anchor=north]{$\ell$};
\fill[yellow] (11.9,0.9) rectangle (12.1,1.1); 
\node[anchor=north] at (12,1) {$\ell$};

%\fill[purple] (13,1) circle (2.5pt) node[black,anchor=north]{$y_1$};

\fill[purple] (12.9,0.933) -- +(0:0.2) -- +(60:0.2) -- +(0,0); 
\node[anchor=north] at (13,1) {$y_1$};

\end{tikzpicture}
\caption{The embeddings in Case 4. The subcases $|S'| \leq |S|-1$ and $|S'|=|S|$ are on the left and right, respectively. The red circle, the yellow square and the purple triangle represent $x$, $y$ and $y'$, respectively.}\label{fig:case4}
\end{figure}
\end{proof}

\section*{Acknowledgements}

The authors thank Robert Johnson and Andrew Treglown for their guidance and useful feedback on a draft of this paper.

\end{document}